\documentclass[final]{siamltex}

\usepackage{epsfig,appendix}
\usepackage[english]{babel}
\usepackage{amssymb,amsmath,amsfonts}
\usepackage{mathrsfs,empheq}
\usepackage{mdwlist,fontawesome}
\usepackage{enumerate,url,enumitem,multirow,multicol}
\usepackage[usenames,dvipsnames,svgnames]{xcolor}

\usepackage{algorithm,extarrows}
\usepackage{algcompatible}
\algnewcommand\algorithmicreturn{\textbf{return}}
\algnewcommand\RETURN{\State \algorithmicreturn}%
\usepackage[outercaption]{sidecap}

\usepackage{tikz}
\usepackage{bm}

\usepackage{pgfplots}
\pgfplotsset{
  tick label style={font=\footnotesize},
  label style={font=\footnotesize},
  legend style={font=\footnotesize}
}

\usepackage{booktabs}
\usepackage{caption,subcaption}

\usepackage{todonotes,bbding,stackengine}

\usepackage[normalem]{ulem} 

\usepackage[colorlinks=true, allcolors=blue]{hyperref} 
\usepackage{cleveref}  

\newcommand{\rsf}{\mathsf{r}}
\newcommand{\ksf}{\mathsf{k}}

\newcommand{\ran}{\mathrm{ran}}

\newcommand{\ransf}{\mathrm{ran}}
\newcommand{\kersf}{\mathrm{ker}}
\newcommand{\intsf}{\mathrm{int}}
\newcommand{\clsf}{\mathrm{cl}}

\newcommand{\slevsf}{\mathrm{slev}}

\newcommand{\Tsf}{\mathsf{T}}

\newcommand{\R}{\mathbb{R}}

\newcommand{\cR}{\mathcal{R}}

\newcommand{\cN}{\mathcal{N}}

\newcommand{\cT}{\mathcal{T}}

\newcommand{\cP}{\mathcal{P}}

\newcommand{\cD}{\mathcal{D}}

\newcommand{\cB}{\mathcal{B}}

\newcommand{\risf}{\mathrm{ri}}
\newcommand{\domsf}{\mathrm{dom}}

\newcommand{\be}{\begin{equation}}
\newcommand{\ee}{\end{equation}}

\DeclareMathOperator{\dom}{dom}         
\DeclareMathOperator{\Arg}{Arg}         
\newtheorem{example}{Example}
\newtheorem{remark}{Remark}
\newtheorem{fact}{Fact}

\newcommand{\infimal}{\triangleright}

\newcommand{\mcup}{\mathbin{\raisebox{-0.20ex} {\scalebox{1.3}{$\cup$}}}} 
 
\newcommand{\mcap}{\mathbin{\raisebox{-0.05ex} {\scalebox{1.3}{$\cap$}}}}

\usepackage[normalem]{ulem} 

\newif\ifcompilePGFfigs
\compilePGFfigstrue

\newcommand{\tabincell}[2]{\begin{tabular}{@{}#1@{}}#2\end{tabular}}

\graphicspath{{figs/}}

\title{Subspace decomposition in regularized least-squares: solution properties, restricted coercivity and beyond
}

\author{Feng Xue\thanks{National key laboratory, Beijing,  China (\tt{fxue@link.cuhk.edu.hk}).}
\and Hui Zhang\textsuperscript{\Envelope}\thanks{Corresponding author. Department of Mathematics, National University of Defense Technology, Changsha, Hunan 410073, China (\tt{h.zhang1984@163.com}).} } 

\date{\today}

\begin{document}

\maketitle

\begin{abstract}
We investigate the solution properties of the regularized least-squares problem. Using a subspace decomposition technique, we derive expressions for the solution set in terms of the conjugate function, from which various properties, including existence, compactness and uniqueness, can then be easily analyzed. A key distinction of our approach from existing works is the separate treatment of existence and compactness. We unify many existing results based on recession cones and sublevel sets, and link them to our findings by connecting the recession function with the recession cone of the subdifferential of the conjugate function. In particular, the concept of restricted coercivity is developed and discussed in various aspects. The associated linearly constrained counterpart is discussed in a similar manner. Its connections to regularized least-squares are further established via the exactness of infimal postcomposition. Our results are supported by  numerous examples, among which the geometric interpretation of the lasso solution deserves further investigations in near future.
\end{abstract}




\section{Introduction}
\subsection{Regularized least squares}
The regularized least-squares problem plays a central role in many application fields such as signal processing, machine learning and statistical inference. It is formulated as \cite{vaiter_thesis,lasso_reload}
\be \label{p1} 
\min_{x\in \R^n} f(x) +\frac{1}{2} \big\| Ax-b \big\|^2,
\ee
where $f:\R^n\rightarrow \R \cup \{+\infty\}$ is a proper, lower semi-continuous (l.s.c.) and convex regularization functions, $A: \R^n \mapsto \R^m$ is a given matrix, $b\in\R^m$ is a given vector. Its associated equality-constrained optimization is given by
\be  \label{p2}
\min_{x\in\R^n} f(x),\ \ \textrm{s.t.\ } Ax=b.
\ee

Numerous works have been devoted to devising and analyzing (first-order) optimization methods for solving \eqref{p1} and \eqref{p2}; see, for example, the comprehensive accounts in \cite{beck_book,ywt_book,glo_book} and more recent overviews in \cite{condat_tour,plc_review_2024}. In this paper, we concentrate on the solution properties of these two problems, rather than on solving algorithms. For \eqref{p1}, the existence and compactness of the solution set have been studied in \cite{teboulle_book,vaiter_thesis} using concepts of coercivity and recession functions. Recent works such as \cite{fadili_cone,he_lasso} have provided dual/conjugate perspectives on the solution properties of \eqref{p1}.
The problem \eqref{p2} has been extensively examined in \cite{fadili_iiima,fadili_cone,venkat_geometry,ame_edge} across various contexts, employing tools such as typical cones and geometric/statistical measures. Connections between \eqref{p1} and \eqref{p2} were established in \cite{fadili_iiima}, where the authors also discussed the strongness and sharpness of the minima.

Regarding the specific case where $f$ is a partly smooth function \cite{lewis}, the solution uniqueness has been thoroughly studied in \cite{vaiter_iiima,fadili_iiima,vaiter_tit_2018}, relying heavily on the concepts of model subspaces and subdifferential decomposability. This class includes many commonly used (sparsity-inducing) regularizers, such as the $\ell_1$-norm. Solution properties of the lasso were analyzed in \cite{lasso_reload,gilbert,zh_lasso,vaiter_tit_2013}, and later extended to the generalized lasso in \cite{vaiter_tit,tib_2012}.
More recently, \cite{he_lasso} investigated the solution uniqueness for gauge functions and examined the lasso solution from a Lagrangian dual perspective. Connections between the solutions of the lasso and basis pursuit problems were discussed in \cite{fadili_iiima} as special cases.

\subsection{Contributions}
Based on a subspace decomposition with respect to $A$---namely, the range space of $A^\top$ and the kernel space of $A$---we establish a characterization of the solution set for the regularized least-squares problem \eqref{p1}, which depends on the subdifferential of the conjugate of $f$. Explicit expressions of the solution set have not been proposed in related works. Previous studies primarily focused on solution existence and uniqueness based on functional properties, without explicitly constructing the solution set. In contrast, our derived solution set provides a straightforward approach to analyze existence, compactness, and uniqueness of solutions. Consequently, our sufficient (and necessary) conditions for the solution properties are also formulated in terms of the subdifferential of the conjugate of $f$. 

Most existing results on solution existence and compactness in \cite{fadili_iiima,fadili_cone,vaiter_thesis} rely on concepts such as recession and coercivity of functions. We unify these terminologies and bridge existing works with our approach by linking the subdifferential of the conjugate of $f$ to the recession function. The restricted coercivity is particularly discussed in various aspects.  For solution uniqueness, we also consolidate several typical cones used in prior studies, such as the tangent cone, radial cone, and recession cone. Our results are shown to be equivalent to \cite[Theorem 3.1]{fadili_cone}, yet our condition is more direct and simpler. Furthermore, our framework encompasses \cite[Theorem 3.8]{he_lasso} as a special case when $f$ is a gauge function, and \cite[Proposition 3.2]{esaim} when $f$ is the $\ell_1$-norm. Note that existing works, such as \cite{teboulle_book,vaiter_thesis}, often treat existence and compactness of the solution set together. If the condition fails, one cannot determine whether the solution set is empty or unbounded. Therefore, existence and compactness are addressed separately in this work.

The equality-constrained problem \eqref{p2} is investigated following an approach very similar to that used for \eqref{p1}. The subspace decomposition also provides a convenient way to derive the solution set, from which we further analyze solution properties. Connections to \eqref{p1} are established through the concept of (exact) infimal postcomposition, extending the results in \cite{fadili_iiima}. Finally, various examples for both problems are provided to illustrate and validate our results, as well as to highlight the limitations of existing ones.

\subsection{Notations and assumptions}
We use standard notations and geometric concepts from convex analysis and variational analysis, which can all be found in the classical and recent monographs \cite{rtr_book,rtr_book_2,urruty,plc_book,teboulle_book}.  

We reserve a number of notations as following. The solution sets of \eqref{p1} and \eqref{p2} are denoted by $X$, and a particular solution is denoted by $x^\star \in X$. The associated residual for \eqref{p1} is denoted as $r:=b-Ax^\star$.  $\iota_C$ denotes an indicator function of the set $C$.  $\cP_C$ stands for an orthogonal projection of a point onto the set $C$. $\ransf A^\top$ and $\kersf A$ denote the range (column) space of $A^\top$ and kernel space of $A$, respectively. The subspace decomposition in this paper is always referred to as the above two orthogonal subspaces, formally written as $x=\cP_{\ransf A^\top} x
+ \cP_{\kersf A} x :=x_\rsf +x_\ksf$. $I_n$ denotes the $n\times n$ identity matrix.  For a (not necessarily maximal) monotone operator $\cB:\R^n \mapsto 2^{\R^n}$, $J_\cB$ denotes a resolvent of $\cB$, which is defined as $J_\cB =(I_n +\cB)^{-1}$; the parallel composition of $\cB$ by $A$ is defined as $A\infimal \cB = (A\circ \cB^{-1} \circ A^\top)^{-1}$ \cite[Definition 25.39]{plc_book}.

\section{Regularized least-squares} \label{sec_rls}
\subsection{Preliminary results}
We first derive the solution set of \eqref{p1} (assumed non-empty for the moment), from which then immediately follow the solution existence and other properties. This idea is fundamentally different from the existing works, which often directly investigated the conditions for solution existence and compactness, without developing the solution set. 

The following results will be useful for deriving the solution set of \eqref{p1}.
\begin{lemma} \label{l_X_p1}
Given the problem \eqref{p1}, $ J_{A\infimal \partial f} $ satisfies the  following.

{\rm (i)} $\ransf J_{A\infimal \partial f} = \ransf ( I_m + A \infimal  \partial f )^{-1} = \domsf
( I_m + A \infimal  \partial f ) =\domsf (A\infimal \partial f)  
=\ransf (A \circ \partial f^* \circ A^\top)   \subseteq \ransf A$. 

{\rm (ii)} $ J_{A\infimal \partial f} = \cP_{\ransf A} \circ  J_{A\infimal \partial f} $, i.e.,  $( I_m + A \infimal  \partial f )^{-1} = 
\cP_{\ransf A} \circ ( I_m + A \infimal  \partial f )^{-1} $.

{\rm (iii)} $J_{A\infimal \partial f} = ( I_m + A \infimal  \partial f )^{-1}$ is at most single-valued.

{\rm (iv)} $ J_{A\infimal \partial f} =  J_{A\infimal \partial f}  \circ \cP_{\ransf A}$, i.e.,  $( I_m + A \infimal  \partial f )^{-1} = 
( I_m + A \infimal  \partial f )^{-1} \circ \cP_{\ransf A}$.

{\rm (v)} $ J_{A\infimal \partial f} =\cP_{\ransf A} \circ   J_{A\infimal \partial f}  \circ \cP_{\ransf A}$.
\end{lemma}
\begin{proof}
(i) Clear by definitions of resolvent and $A\infimal \partial f$.

(ii) immediately follows from (i).

(iii) If $b\notin \ransf ( I_m + A \infimal  \partial f)$, we have $( I_m + A \infimal  \partial f )^{-1} (b) =\varnothing$. If  $b\in \ransf ( I_m + A \infimal  \partial f)$, the uniqueness of $( I_m + A \infimal  \partial f)^{-1} (b)$ can be shown  by the monotonicity of $A\infimal \partial f$. Indeed, let $\{a_1, a_2\}\subseteq ( I_m + A \infimal  \partial f)^{-1} (b)$,  then, $b\in a_1+ (A\infimal \partial f) (a_1)$, and $b\in a_2+ (A\infimal \partial f) (a_2)$. Subtracting yields $-(a_1-a_2) \in (A\infimal \partial f)(a_1) - (A\infimal \partial f)(a_2)$, which further leads to
\[
0\le \|a_1-a_2\|^2 = \langle a_1-a_2| a_1-a_2\rangle  \in 
- \big\langle (A\infimal \partial f)(a_1) - (A\infimal \partial f)(a_2)
\big| a_1-a_2 \big\rangle \le 0,
\]
which is $a_1=a_2$. 

(iv) By (iii),  $\forall b\in \ransf ( I_m + A \infimal  \partial f)$,  we deduce that $a=(I_m + A \infimal  \partial f)^{-1} (b) \Longleftrightarrow b\in a+ (A \infimal  \partial f) (a)
 \Longleftrightarrow b - a \in (A \infimal  \partial f) (a)
  \Longleftrightarrow a\in  (A \circ \partial f^* \circ A^\top) (b-a)$.
  Combining with $a\in \ransf A$ by (i), it is further equivalent to $a\in  (A \circ \partial f^* \circ A^\top) (\cP_{\ransf A} b-a)$, which yields (iii).
  
(v). In view of (ii) and (iv).
\end{proof}

Then, it is ready to develop expressions for the solution set of \eqref{p1}.
\begin{theorem} \label{t_X_p1}
The solution set $X$ of \eqref{p1} can be expressed in the following forms.

{\rm (i)} $X = (\partial f + A^\top A)^{-1}  A^\top  b$.

{\rm (ii)}  $X =x_\rsf^\star + \big( \big(\partial f^* ( A^\top r )-  x_\rsf^\star \big) \mcap \kersf A \big)$, where
 $x_\rsf^\star = A^\dagger ( I_m + A\infimal \partial f )^{-1}   (b) $.

{\rm (iii)} $X =x^\star +  \big( \big(\partial f^* ( A^\top r )-  x^\star \big) \mcap \kersf A \big)$, if a particular solution $x^\star\in X$ is given. 
\end{theorem}
\begin{proof}
(i) We develop:
\begin{eqnarray} \label{a1}
&& x^\star \in X := \Arg\min_{x\in\R^n} f(x)+\frac{1}{2}
\|Ax-b\|^2 
\nonumber \\
 & \  \Longleftrightarrow \ &
0\in \partial f(x^\star) + A^\top (Ax^\star-b) 
\quad \textrm{(by Fermat's rule \cite[Theorem 16.3]{plc_book})}
\nonumber \\
& \  \Longleftrightarrow \ &
0 \in (\partial f+ A^\top A) x^\star - A^\top  b 
\nonumber \\
& \  \Longleftrightarrow \ &
x^\star  \in X =  (\partial f + A^\top A)^{-1}  A^\top  b .
\nonumber 
\end{eqnarray}

(ii) To begin with Fermat's rule,  we have:
\begin{eqnarray} \label{a1}
&& 0\in \partial f(x^\star) + A^\top (Ax^\star-b) 
\nonumber \\
& \  \Longleftrightarrow \ &
x^\star \in (\partial f)^{-1} 
\big(A^\top (b-Ax^\star) \big) 
\nonumber \\
& \  \Longleftrightarrow \ &
x^\star \in  \partial f^* 
\big(A^\top (b-Ax^\star) \big) \quad 
\textrm{by $(\partial f)^{-1} = \partial f^*$ when  $f\in\Gamma_0(\R^n)$}
\nonumber \\
& \  \Longleftrightarrow \ &
x_\rsf^\star +  x_\ksf^\star\in (\partial f^* \circ 
A^\top ) \big(  b-Ax_\rsf^\star \big) \quad 
\textrm{by $x^\star = x_\rsf^\star +x_\ksf^\star$ and $Ax^\star = Ax_\rsf^\star$}
\nonumber \\
& \  \Longleftrightarrow \ &
\left\{ \begin{array}{lll}
Ax_\rsf^\star &\in & (A\circ \partial f^* \circ 
A^\top ) \big(  b-Ax_\rsf^\star \big); \\
 x_\ksf^\star &\in &  (\partial f^* \circ 
A^\top ) \big(  b-Ax_\rsf^\star \big) - x_\rsf^\star,\textrm{\ and\ }
x_\ksf^\star \in \kersf A.
\end{array}\right. \\
& \  \Longleftrightarrow \ &
\left\{ \begin{array}{lll}
b &\in & \big( I_m+ (A\circ \partial f^* \circ 
A^\top )^{-1} \big)  (Ax_\rsf^\star) ; \\
 x_\ksf^\star &\in &  \big( (\partial f^* (A^\top r ) - x_\rsf^\star \big) \mcap  \kersf A \quad 
 \textrm{by definition of $r=b-Ax_\rsf^\star$}.
\end{array}\right.
\nonumber \\
& \  \Longleftrightarrow \ &
\left\{ \begin{array}{lll}
Ax_\rsf^\star &= & \big( I_m+ (A\infimal \partial f \big)^{-1}  (b) ; 
\quad \textrm{by def. of $A\infimal \partial f$ and Lemma \ref{l_X_p1}-(ii)} \\
 x_\ksf^\star &\in & \big( (\partial f^* (A^\top r ) - x_\rsf^\star \big) \mcap  \kersf A.
 \nonumber
\end{array}\right.
\end{eqnarray}
from which follows (ii). 

(iii) If a particular solution $x^\star$ is provided, one must have by (ii) that $\cP_{\ransf A^\top} x^\star = x_\rsf^\star = A^\dagger \big( I_m + A\infimal \partial f \big)^{-1}   (b)$, and $\cP_{\kersf A} x^\star \in (\partial f^*(A^\top r)-x_\rsf^\star) \mcap \kersf A$. Then by (ii), we develop:
\begin{eqnarray}
X &=&\cP_{\ransf A^\top} x^\star  + \big( \big(\partial f^* ( A^\top r )-  \cP_{\ransf A^\top} x^\star \big) \mcap \kersf A \big)
\nonumber \\
 &=&\cP_{\ransf A^\top} x^\star  + \cP_{\kersf A} x^\star + \big( \big(\partial f^* ( A^\top r )-  \cP_{\ransf A^\top} x^\star - \cP_{\kersf A} x^\star \big) \mcap (\kersf A-\cP_{\kersf A} x^\star) \big)
\nonumber \\
 &=& x^\star  + \big( \big(\partial f^* ( A^\top r )-  x^\star  \big) \mcap  \kersf A  \big).
 \nonumber
\end{eqnarray}
\end{proof}
\begin{remark}
{\rm (i)} The key for applying Fermat's rule in the proof of Theorem \ref{t_X_p1}-(i) is that $\partial (f  +\frac{1}{2}\|A\cdot-b\|^2) = \partial f + \partial \big( \frac{1}{2}\|A\cdot-b\|^2 \big) $ by \cite[Theorem 16.47 and Corollary 16.48]{plc_book}, due to $\domsf \big( \frac{1}{2}\|A \cdot - b\|^2 \big)  = \R^n$. 

{\rm (ii)}  The monotonicity of $A\circ \partial f^* \circ  A^\top $  is not necessarily maximal, and so is $A\infimal \partial f$. Without maximality, the resolvent $J_{A\infimal \partial f}$  may not be of full domain. That is also why there is a possibility of $b\notin \ransf ( I_m +  A\infimal \partial f )$ in the proof of Lemma \ref{l_X_p1}-(ii).

{\rm (iii)} If   $b\in \ransf ( I_m +  A\infimal \partial f )$, the existence of $Ax_\rsf^\star$ has been ensured by Lemma \ref{l_X_p1}-(i).
\end{remark}

Further results are presented below.
\begin{proposition} \label{p_X_p1}
Defining $\cT = (\partial f+   A^\top A)^{-1} A^\top$, the following hold.

{\rm (i)} $\cT = \cT \circ \cP_{\ransf A}$.

{\rm (ii)} $X\ne \varnothing$, if and only if $x_\rsf^\star$ defined in Theorem \ref{t_X_p1}-(ii) exists.

{\rm (iii)} $\domsf \cT = \ransf (I_m +A\infimal \partial f)$.

{\rm (iv)} $\domsf \cT = \domsf \cT +\kersf A^\top = \cP_{\ransf A} (\domsf \cT) +\kersf A^\top$.
\end{proposition}
\begin{proof}
(i) is due to $A^\top = A^\top \circ \cP_{\ransf A}$.

(ii) Equivalently, we now show that $x_\ksf^\star$ exists, as long as $x_\rsf^\star$ exists. By Theorem \ref{t_X_p1}-(ii), it is equivalent to say $\big(\partial f^* ( A^\top r )-  x_\rsf^\star \big) \mcap \kersf A  \ne \varnothing$, if $x_\rsf^\star$ exists. This expression can be translated as: the projection of the set $\partial f^* ( A^\top r) $ onto $\ransf A^\top$ must (at least) cover the point $x_\rsf^\star$, i.e., $x_\rsf^\star \in \cP_{\ransf A^\top} \big(\partial f^* ( A^\top r ) \big) $. This is obviously true, because  $Ax_\rsf^\star \in (A \circ \partial f^*) ( A^\top r )$ by the first line of \eqref{a1}. 

(iii) immediately follows from (ii).

(iv) is from (i), or by combining (iii) and Lemma \ref{l_X_p1}-(iv).
\end{proof}

\subsection{Solution existence}
Equipped with the preliminary results, it is easy to obtain a number of sufficient (and necessary) conditions for solution existence. First, the solution existence should depend on the specific given $b\in\R^m$.
\begin{lemma} \label{l_ex_p1}
Regarding  \eqref{p1}, $X\ne\varnothing$, if and only if
$b\in \ransf( I_m +   A\infimal \partial f )$. 
\end{lemma}
\begin{proof}
Clear by Theorem \ref{t_X_p1}-(ii), Proposition \ref{p_X_p1}-(ii) and (iii). More importantly, this condition has implied the basic viability of $\ransf \partial f\mcap \ransf A^\top \ne \varnothing$, which guarantees the viability of $A\circ \partial f^*\circ A^\top$ and $\domsf (A\infimal \partial f) \ne \varnothing$. 
\end{proof}

From this follow the following key results.
\begin{theorem} [Solution existence]  \label{t_ex_p1}
Regarding \eqref{p1}, $X\ne\varnothing$ for any $b\in \R^m$,

{\rm (i)} if and only if $\domsf\cT = \ransf (I_m+A\infimal \partial f) = \R^m$;

{\rm (ii)} if and only if $\ransf (\partial f+   A^\top A) \supseteq \ransf A^\top$;

{\rm (iii)} if and only if $A\circ \partial f^* \circ   A^\top:\R^m\mapsto 2^{\R^m}$ is maximally monotone;

{\rm (iv)} if  $0\in \risf \  (\ransf \partial f - \ransf A^\top ) $ or equivalently, $(\risf\ \ransf \partial f) \mcap \ransf A^\top \ne \varnothing$;

{\rm (v)} if  $0\in \risf \   \ransf \partial f$.
\end{theorem}
\begin{proof}
(i) Clear by  Theorem \ref{t_X_p1}-(i) and Proposition \ref{p_X_p1}-(iii), or Lemma \ref{l_ex_p1}.

(ii) Clear by Theorem \ref{t_X_p1}-(i) or definition of $\cT$.

(iii) Based on the monotonicity of $A\infimal \partial f$ (due to $f\in \Gamma_0(\R^n)$ and \cite[Proposition 20.10]{plc_book}),  we develop, $\forall b\in\R^m$, that:
\begin{eqnarray} 
&& X\ne\varnothing,\ \forall b\in \R^m
\nonumber \\
&  \Longleftrightarrow &
\ransf ( I_m + A\infimal \partial f ) = \R^m
\quad \textrm{by (i)}
\nonumber \\
&  \Longleftrightarrow &
A \infimal  \partial f \textrm{ is maximally monotone (by Minty's theorem \cite[Theorem 21.1]{plc_book})}
\nonumber \\
& \  \Longleftrightarrow \ &
A\circ \partial f^* \circ A^\top \textrm{ is maximally monotone (by  \cite[Propopsition 20.22]{plc_book})}
\nonumber
\end{eqnarray}

(iv) The first condition follows from (ii), in view of \cite[Corollary 25.6 or Proposition 25.41-(iv)]{plc_book}. We now show the equivalence with the second condition under the assumption of $\ransf \partial f \mcap \ransf A^\top \ne\varnothing$. Since $\ransf A^\top$ is a linear subspace,  we develop  by \cite[Corollary 6.6.2]{rtr_book} that: 
\begin{eqnarray}
\risf \  (\ransf \partial f - \ransf A^\top ) 
&= &  \risf \  (\ransf \partial f + \ransf A^\top ) 
\nonumber \\
& = &  \risf \ \ransf \partial f +\risf\ \ransf A^\top  
\quad \textrm{by \cite[Corollary 6.6.2]{rtr_book}}
\nonumber \\
& = &  \risf \  \ransf \partial f + \ransf A^\top 
\nonumber \\
& = &  \risf \  \ransf \partial f - \ransf A^\top.
\nonumber
\end{eqnarray}
Thus, $0\in \risf \  (\ransf \partial f - \ransf A^\top ) $ is equivalent to  
$0 \in \risf \  \ransf \partial f - \ransf A^\top$, i.e., $\risf \  \ransf \partial f \mcap \ransf A^\top \ne\varnothing$. This can also be proved by \cite[Eq. (III-2.1.5)]{urruty}. Under this condition, we further have by  \cite[Proposition III-2.1.10]{urruty} or \cite[Corollary 6.5.1]{rtr_book} that $ \risf \  (\ransf \partial f \mcap \ransf A^\top )  =  \risf \ \ransf \partial f \mcap \ransf A^\top$. 
 
(v) According to the second condition of (iv), since $0 \in  \ransf A^\top$, (v) is sufficient for $ \ransf A^\top \mcap \risf\  \ransf \partial f \supseteq \{0\} \ne \varnothing$.
\end{proof}

We have a few remarks here. 
\begin{enumerate}
  \item The implication of above four conditions is (v)$\Longrightarrow$(iv)$\Longrightarrow$(iii)$\Longleftrightarrow$(ii)$\Longleftrightarrow$(i). The condition (ii) was mentioned in \cite{plc_warped}. However, it is of very little interest, because it does not decouple the roles of $f$ and $A$, and fails to reveal any interplay between $f$ and $A$. We focus on (iii)--(v) in this paper.

\item  The solution existence of \eqref{p1} does not require  $0\in \ransf \partial f$, e.g., $f(x)=x$ or $e^x$. Hence, the stronger condition (iv) is usually not satisfied.

\item Again, the viability condition of $\ransf \partial f\mcap \ransf A^\top \ne\varnothing$ is also implied by the maximal monotonicity of $A\circ \partial f^* \circ A^\top$. 
\end{enumerate}

The solution to \eqref{p1} has the following well-known properties. 
\begin{corollary} \label{c_unique}
If $X\ne\varnothing$,  the following hold.

{\rm (i)} $Ax^\star$ is unique for any $x^\star \in X$,  which is given as
$Ax^\star = ( I_m +  A\infimal  \partial f )^{-1}  (b)$.

{\rm (ii)}  $x_\rsf^\star$ is unique.

{\rm (iii)} The residual $r: = b-A x^\star$ is unique.

{\rm (iv)} $f(x^\star)$ is unique.

{\rm (v)} $X$ is closed and convex. 
\end{corollary}
\begin{proof}
(i): In view of Lemma \ref{l_X_p1}-(iii).

(ii)--(iv): Clear from (i).

(v) \cite[Proposition 16.4-(iii)]{plc_book}.
\end{proof}

The uniqueness of $Ax^\star$ has been presented in \cite[Proposition 3-(1)]{vaiter_maximal}, \cite[Lemma 2]{vaiter_low} and \cite[lemma 4.1]{zh_lasso} for lasso problem. This is now extended to general $f\in\Gamma_0(\R^n)$. 

Be aware of  that $ A\circ \partial f^* \circ   A^\top$ cannot be simplified to
$\partial (f^* \circ   A^\top)$ generally, since $ A\circ \partial f^* \circ   A^\top \subseteq  \partial (f^* \circ   A^\top)  $  by \cite[Proposition 16.6]{plc_book}. Under Theorem \ref{t_ex_p1}-(iv) or (v), it then holds that $ A\circ \partial f^* \circ   A^\top =  \partial (f^* \circ   A^\top)  $, and thus,   $Ax^\star$ can be written as $Ax^\star =  \big( I_m +    (\partial (f^* \circ   A^\top) )^{-1}  \big)^{-1}  (b) = \big( I_m +  \partial (f^* \circ   A^\top)^*\big)^{-1}  (b) $, where the last equality is due to $f^*\circ A^\top \in \Gamma_0(\R^m)$ as mentioned above.

It is emphasized that our conditions do not guarantee the compactness of the solution set. The existence and compactness are discussed separately in this paper. The existing works, which often tackle both problems together, will be discussed and compared in Sect. \ref{sec_cmp_p1}.

\subsection{Compactness} \label{sec_cmp_p1}
\subsubsection{Our results}
It is then natural to seek the conditions, under which the solution set $X$ of \eqref{p1} is bounded or reduces to a singleton. Let us first present  several nice properties of recession cone, and then develop the compact property.
\begin{lemma} \label{l_recession}
Let $D$ be a proper linear subspace of $\R^n$, and  $C\subset \R^n$ be a closed and convex set of $\R^n$ such that $C\mcap D \ne \varnothing$, the following hold.

{\rm (i)} $(C\mcap D)_{\infty} = C_{\infty} \mcap D$;

{\rm (ii)}  $(C-x_0)_{\infty} = C_{\infty}$,  $\forall x_0\in C$;

{\rm (iii)} $C\mcap D \subseteq \cP_{D} (C)$.
\end{lemma}

The proof is postponed to Appendix \ref{app_1}.

\begin{proposition} [Solution compactness]  \label{p_cmp_p1}
The solution set $X$ of \eqref{p1} is com-pact,  

{\rm (i)} if and only if $\big(  \partial f^* (   A^\top  r)   \big)_{\infty} \mcap \kersf A = \{0\} $, where $\big(  \partial f^* (   A^\top  r)   \big)_{\infty}$ denotes the recession cone of 
$\partial f^* (   A^\top  r)$.

{\rm (ii)} if $ \big( \cP_{\kersf A} (\partial f^* (   A^\top  r)  ) \big)_{\infty} = \{0\}$. 
\end{proposition}
\begin{proof}
(i) Denote $C=  \partial f^* (   A^\top  r)  -x^\star$ for brevity, which satisfies $0\in C\mcap \kersf A \ne\varnothing$. We then develop:
\begin{eqnarray}
&& \textrm{compactness of } X 
\nonumber \\
& \  \Longleftrightarrow \ & \textrm{compactness of }
C \mcap \kersf A \quad \textrm{[by Theorem \ref{t_X_p1}-(iii)]}
\nonumber \\
& \  \Longleftrightarrow \ & \textrm{boundedness of }
C \mcap \kersf A\quad \textrm{[by Corollary \ref{c_unique}-(v)]}
\nonumber \\
& \  \Longleftrightarrow \ &
(C \mcap \kersf A)_{\infty}=\{0\} \quad \textrm{[by \cite[Proposition III-2.2.3]{urruty}]}
\nonumber \\
& \  \Longleftrightarrow \ &
C_{\infty} \mcap \kersf A = \{0\} \quad \textrm{[by Lemma \ref{l_recession}-(i) and $0\in C\cap \kersf A$]}
\nonumber \\
& \  \Longleftrightarrow \ &
 \big(  \partial f^* (   A^\top  r) \big)_{\infty} \mcap \kersf A = \{0\} \quad \textrm{[by Lemma \ref{l_recession}-(ii)]}
\nonumber
\end{eqnarray}

(ii) Since $C\mcap \kersf A\subseteq \cP_{\kersf A} (C)$ by Lemma \ref{l_recession}-(iii), and then it is clear by definition that
 $(C\mcap \kersf A)_{\infty} = \big(  \partial f^* (   A^\top  r) \big)_{\infty} \mcap \kersf A \subseteq (\cP_{\kersf A} (C))_{\infty}$. It suffices to let $ (\cP_{\kersf A} (C))_{\infty} = \{0\}$.  Last, note that $ (\cP_{\kersf A} (C))_{\infty}  = (\cP_{\kersf A} (\partial f^* (   A^\top  r)  -x^\star))_{\infty}   = (\cP_{\kersf A} (\partial f^* (   A^\top  r))  -  x_\ksf^\star)_{\infty}  =  (\cP_{\kersf A} \big( \partial f^* (   A^\top  r) ) \big)_{\infty}$, where the last equality is due to Lemma \ref{l_recession}-(ii).
\end{proof}

Proposition \ref{p_cmp_p1}-(ii) is sufficient but not necessary for solution compactness.  This will be shown by Example \ref{eg_lasso} later.

\subsubsection{Developments of existing results}
Many concepts related to solution compactness have been proposed in literature, e.g., \cite{vaiter_thesis,lasso_reload,fadili_cone,fadili_iiima}. In this section, we are going to investigate the relations among them, and further develop the existing results. The following concepts are commonly used for analyzing solution compactness in literature, e.g., \cite{vaiter_thesis,fadili_iiima,venkat_geometry}. 
\begin{definition} \label{def_cmp}
{\rm (i) (restricted injectivity) \cite[Definition 5.3]{vaiter_thesis}} A linear mapping $A$ is injective on a set $C\subset\R^n$ such that $0\in C$, if $Av=0$ with $v\in C$ yields $v=0$.

{\rm (ii) (restricted coercivity)} A function $f\in\Gamma_0(\R^n)$ is coercive on a proper linear subspace $D$, if $\lim_{\|\cP_D x\|\rightarrow +\infty}
f(x) =+\infty$, where $\cP_D x$  denotes the projection of $x$ onto $D$. 

{\rm (iii) (recession cone)} Given $f\in\Gamma_0(\R^n)$, the recession cone of $f$ is defined as $\cR_f = \big\{ d\in\R^n:  f(x+t d) \le f(x), \ \forall t > 0,\ \forall x\in \domsf f \big\}$.

{\rm (iv) (sublevel set) \cite[Definition 2.5]{vaiter_thesis}} A sublevel set of $f$ passing through $x_0$ is defined as  $\mathrm{slev}_{x_0} f:=\big\{z\in\R^n:  f(z)\le f(x_0) \big\}$.

{\rm (v) (recession direction) \cite[Definition 3.1.2]{teboulle_book}} The recession direction of $f$ consists of all directions $d\in\R^n$ satisfying $f_{\infty} (d) \le 0$, where $f_{\infty}$ denotes the recession function of $f$.
\end{definition}

The concepts in Definition \ref{def_cmp} can be connected and unified as follows.
\begin{lemma} \label{l_cmp}
Given $f\in\Gamma_0(\R^n)$, the following hold.

{\rm (i)} The restricted injectivity of $A$ on a set $C\subset\R^n$ containing 0 can also be defined as  $\kersf A\mcap C = \{0\}$, i.e.,  the only vector in $C$ that $A$ maps to the zero vector is the zero vector itself.

{\rm (ii)} $\cR_f = \big\{ d\in\R^n:  f_{\infty} (d) \le 0 \big\}$, i.e, the recession cone of $f$ is exactly the set of recession directions.

{\rm (iii)}  $\kersf f_{\infty} \subseteq \cR_f = (\slevsf_{x} f )_{\infty}$, $\forall x\in\domsf f$.

{\rm (iv)} If $\inf f>-\infty$,  then $f_{\infty}(d) \ge 0$ ($\forall d\in\R^n$), and $\kersf f_{\infty} = \cR_f$. 

{\rm (v)} If $\displaystyle \inf_{\substack{ t>0    \\ d \in \kersf A } }  f(x+td)>-\infty$ ($\forall x\in\domsf f$),  then  $\kersf f_{\infty}\mcap \kersf A = \cR_f \mcap \kersf A$. 
\end{lemma}
\begin{proof}
(i) clear.  

(ii) First, if $d\in\cR_f$, then $f(x+td)\le f(x)$, $\forall x\in\domsf f$ and $\forall t>0$. This implies that $\sup_{t>0} \frac{f(x+td)-f(x)}{t} \le 0$, which in turn yields $f_{\infty}(d) \le 0$ by definition of $f_{\infty}$. The converse statement is also true by the same argument. 

(iii) It follows that  $\kersf f_{\infty} =  \big\{ d\in\R^n:  f_{\infty} (d) = 0 \big\}
 \subseteq  \big\{ d\in\R^n:  f_{\infty} (d) \le 0 \big\} = \cR_f$, where the last equality is due to (i).

Now let us show $\cR_f = (\slevsf_{x} f )_{\infty}$, $\forall x\in\domsf f$. First, we claim that given $x_0\in\domsf f$, it holds that
\be \label{h2}
(\slevsf_{x_0} f )_{\infty} = \big\{ d\in\R^n: f(x_0+td) \le f(x_0), \ \forall t>0 \big\}.
\ee
Indeed, by definition of recession cone, we have $(\slevsf_{x_0} f )_{\infty} = \big\{ d\in\R^n: z+td \in  \slevsf_{x_0} f, \ \forall t>0,\ \forall z\in \slevsf_{x_0} f  \big\}$. Combining with the definition of $\slevsf_{x_0} f$, it becomes:
\[
(\slevsf_{x_0} f )_{\infty} = \big\{ d\in\R^n: f(z+td) \le f(x_0), \ \forall t>0,\ \forall z \textrm{\ satisfying\ } f(z) \le f(x_0)  \big\}.
\]
Since the recession cone does not depend on specific $z$, taking $z=x_0$ proves our claim of \eqref{h2}. Moreover, the  proof also shows that $x_0$ can be arbitrary point in $\domsf f$. Last, observe that \eqref{h2} exactly coincides with definition of $\cR_f$ (cf. Definition \ref{def_cmp}-(iii)).

(iv) If $\inf f>-\infty$ , then by definition we have, $\forall d\in\R^n$, $\forall x\in \domsf f$, that
\[
f_{\infty} (d) = \lim_{t\rightarrow +\infty} 
\frac{f(x+td)-f(x)}{t}.
\] 
Since $f(x+td) >-\infty$ as $t\rightarrow +\infty$, $f_{\infty}(d)$ cannot be negative, i.e., $f_{\infty}(d) \ge 0$. Under this prerequisite, it then follows that $\kersf f_{\infty} = \cR_f$.

(v) It is clear from (iii) that $\kersf f_{\infty}\mcap \kersf A \subseteq \cR_f \mcap \kersf A$. To show $\kersf f_{\infty}\mcap \kersf A \supseteq \cR_f \mcap \kersf A$, we first claim that $f_{\infty} (d) = \lim_{t\rightarrow +\infty} \frac{f(x+td)-f(x)} {t} \ge 0$, $\forall d\in \kersf A$, if  $\displaystyle \inf_{\substack{ t>0    \\ d \in \kersf A } }  f(x+td)>-\infty$. Take $d\in \cR_f \mcap \kersf A$, then $f_{\infty}(d) \le 0$ and $d\in \kersf A$. Combining with  $f_{\infty} (d)  \ge 0$, $\forall d\in \kersf A$, we have $f_{\infty}(d) = 0$, $\forall d\in \kersf A$, which shows $d\in \kersf f_{\infty}$.
\end{proof}

\begin{remark}
{\rm (i)} $\cR_f$ contains all directions along which $f$ cannot grow to $+\infty$, and  $\cR_f \ne \kersf f_{\infty}$ for general $f\in\Gamma_0(\R^n)$. Taking $f(x)=x$ for example, where $\kersf f_{\infty} = \{0\} \subset  (-\infty,0] =\cR_f$. However, $\cR_f =\kersf f_{\infty}$ does not necessarily imply $\inf f>-\infty$, considering $f=-\log$, which satisfies $\cR_f=\kersf f_{\infty}=[0,+\infty)$ and $\inf (-\log) = -\infty$.

{\rm (ii)} The converse statement of Lemma \ref{l_cmp}-(iv) is not true, i.e.,  $f_{\infty}(d)\ge 0$ ($\forall d\in\R^n$) does not necessarily imply $\inf f>-\infty$. Also, if $\inf f=-\infty$, it does not necessarily hold that $f_{\infty}(d) <0$ for some $d$, e.g., $f=-\log$ defined on $(0,+\infty)$. 
\end{remark}

Our proposed {\it restricted coercivity}, as a new definition, will be a focus here. Its implications are as follows. According to \cite[Definition 3.1.1]{teboulle_book}, Lemma \ref{l_coercive}-(iv) can also be regarded as a definition of restricted coercivity. 
\begin{lemma} \label{l_coercive}
If $f \in\Gamma_0(\R^n)$ is coercive on a proper linear subspace $D\subset \R^n$, then the following hold.

{\rm (i)} $\lim_{t\rightarrow +\infty} f(x_0+td) =+\infty$, $\forall x_0\in \domsf f$, $\forall d\in D$;

{\rm (ii)} $\liminf_{\| \cP_D x\|\rightarrow \infty} \frac{f(x)}{\|\cP_D x\|}>0$;

{\rm (iii)} $\exists \alpha>0, \beta\in\R$, such that $f(x) \ge \alpha \|\cP_D x\| +\beta$;

{\rm (iv)} $f_{\infty}(d) >0$, $\forall d\in D\backslash \{0\}$.
\end{lemma}
\begin{proof}
(i)--(iii): \cite[Proposition 14.16]{plc_book} and \cite[Proposition 3.1.2]{teboulle_book}.

(iv) By definition of recession function \cite[Proposition 2.4]{vaiter_thesis}, we have, $\forall d\in D$, 
\[
f_{\infty}(d) = \lim_{t\rightarrow+\infty} \frac{f(x+td)-f(x)}{t}
\ge   \lim_{t\rightarrow+\infty} 
\frac{ \alpha \|\cP_D (x+td)\| +\beta-f(x)}{t}
= \alpha \|d\|.
\]
\end{proof}

We summarize the sufficient and necessary conditions for solution existence and compactness as below, which include several existing results. 
\begin{proposition} \label{p_cmp}
The solution set of \eqref{p1} is non-empty and compact, if and only if any of the following equivalent conditions holds.

{\rm (i)}  $f+\frac{1}{2}\|A\cdot-b\|^2$ is coercive;

{\rm (ii)}  sublevel set, i.e., $\mathrm{slev}_\xi (f+\frac{1}{2}\|A\cdot-b\|^2)$ for any $\xi\in\R$,  is bounded; 

{\rm (iii)}  $0\in \intsf\ \domsf (f+\frac{1}{2}\|A\cdot-b\|^2)^*$;

{\rm (iv)} $f$ is coercive on $\kersf A$;

{\rm (v)} $\lim_{t\rightarrow \infty} f(x_0+td) =+\infty$, $\forall d\in \kersf A\backslash \{0\}$, $\forall x_0\in\domsf f$;

{\rm (vi)} $f_{\infty}(d) >0$, $\forall d\in \kersf A\backslash \{0\}$; 

{\rm (vii)} $\cR_f \mcap \kersf A=\{0\}$, i.e., $A$ is injective on $\cR_f$;

{\rm (viii)} $\kersf f_{\infty} \mcap \kersf A=\{0\}$, if $\inf f>-\infty$ {\rm \cite[Lemma 5.1]{vaiter_thesis}};

{\rm (ix)}  $\kersf f_{\infty} \mcap \kersf A=\{0\}$, if 
$\displaystyle \inf_{\substack{ t>0    \\ d \in \kersf A } }  f(x+td)>-\infty$, $\forall x\in\domsf f$.
\end{proposition}

\begin{proof}
(i)--(iii): by \cite[Proposition 14.16]{plc_book} and \cite[Proposition 3.1.3]{teboulle_book}.

(iv)--(vi): Lemma \ref{l_coercive}.

(i)$\Longleftrightarrow$(v): coercivity of $f+\frac{1}{2}\|A\cdot-b\|^2$ implies $\lim_{t\rightarrow +\infty} f(x+td) +\frac{1}{2}
\|A(x+td)-b\|^2 =+\infty$, $\forall d\in\R^n$. If $d\in \kersf A$, it becomes  $\lim_{t\rightarrow +\infty} f(x+td) +\frac{1}{2}
\|Ax-b\|^2 =+\infty$, which indicates that  $\lim_{t\rightarrow +\infty} f(x+td) =+\infty$, with $d\in\kersf A$. This is (v). Conversely, if $d\in\kersf A$, since  $\lim_{t\rightarrow +\infty} f(x+td) =+\infty$, we then have $\lim_{t\rightarrow +\infty} f(x+td) +\frac{1}{2} \|A(x+td)-b\|^2 =+\infty$, $\forall d\in \kersf A$. If $d\notin \kersf A$, then $d_\rsf \ne 0$. As $t\rightarrow +\infty$, $\frac{1}{2}\|tAd\|^2 = \frac{t^2}{2}\|Ad_\rsf\|^2$ dominates $\frac{1}{2}\|A(\cdot +td)-b\|^2$ and further $f(\cdot+td)+\frac{1}{2}\|A(\cdot +td)-b\|^2$. Thus, 
 $\lim_{t\rightarrow +\infty} f(x+td) +\frac{1}{2}
\|A(x+td)-b\|^2 =+\infty$, $\forall d\notin\kersf A$. Combining both cases yields (i). 

(vi)$\Longleftrightarrow$(vii): Both directions are proved by contradiciton. First, we assume that $(\cR_f \mcap \kersf A) \backslash \{0\} \ne \varnothing$, then $\exists d_0 \ne 0$, such that $d_0\in \cR_f \mcap \kersf A$. This indicates, by Lemma \ref{l_cmp}-(i),  that $d_0\in\kersf A$, such that $f_{\infty}(d_0) \le 0$.  Conversely, we assume by contradiction that $\exists d_0 \in \kersf A\backslash \{0\}$, such that $f_{\infty}(d_0) \le 0$. It implies that $0\ne d_0 \in \cR_f\mcap \kersf A$, by Lemma \ref{l_cmp}-(i). 

(vii)$\Longleftrightarrow$(viii): Lemma \ref{l_cmp}-(iii), which also shows that the condition of $f$ being bounded below is necessary for (viii). 

(ix) Combine (vii) with  Lemma \ref{l_cmp}-(v). 
 \end{proof}

Regarding a recent result of \cite[Prop. 3.1-(a)]{lasso_reload}, we stress that the condition of $f_\infty(d)>0$, $\forall d\in \kersf A\backslash \{0\}$ is sufficient but not necessary for solution existence of \eqref{p1}.  Indeed, if this condition is violated, the solution set may still be non-empty, but unbounded. Considering the case of  $f:\R^2\mapsto \R: (x_1,x_2)\mapsto \max\{0,x_1\}$, $A=\begin{bmatrix}
0 & 1 \end{bmatrix}$, we obtain $f_{\infty}(d_1,d_2) = 0$, whenever $(d_1,d_2)\in (-\infty,0)\times \{0\}\subset \kersf A$. This condition is not fulfilled, but $X=(-\infty, 0] \times \{1\} \ne\varnothing$ (unbounded). 

Regarding the condition of $\kersf f_{\infty} \mcap \kersf A=\{0\}$, 
(ix) extends (viii) to a milder assumption that  $\displaystyle \inf_{\substack{ t>0    \\ d \in \kersf A } }  f(x+td)>-\infty$, $\forall x\in\domsf f$. This does not require $\inf f>-\infty$. Moreover, this condition was also mentioned in \cite[Sect. II-A]{fadili_2013} for $f$ being a decomposible norm (always non-negative). For this condition,  we further have the following
\begin{corollary} \label{c_cmp_p1}
To proceed with Proposition \ref{p_cmp}-(viii), the following hold. 

{\rm (i)} If $X$ is non-empty and compact, then $\kersf f_{\infty}\mcap \kersf A = \{0\}$. 

{\rm (ii)} If $\kersf f_{\infty}\mcap \kersf A = \{0\}$,  $X$ is compact. 

{\rm (iii)} $X=\varnothing$, if  $\displaystyle \inf_{\substack{ t>0    \\ d \in \kersf A } }  f(x+td) = -\infty$, $\forall x\in\domsf f$. 

{\rm (iv)} $\displaystyle \inf_{\substack{ t>0    \\ d \in \kersf A } }  f(x+td) > -\infty$ ($\forall x\in\domsf f$), if $X\ne \varnothing$. 
\end{corollary}
\begin{proof}
(i)  If $X$ is non-empty and compact, then $\cR_f\mcap \kersf A=\{0\}$ by Proposition \ref{p_cmp}-(vii). Thus,  $\kersf f_{\infty}\mcap \kersf A = \{0\}$ by Lemma \ref{l_cmp}-(ii). 

(ii) If $\kersf f_{\infty}\mcap \kersf A = \{0\}$, it implies that $f_{\infty}(d) \ne 0$, $\forall d\in \kersf A\backslash\{0\}$. If $\exists 0\ne d_0\in \kersf A$ such that $f_{\infty}(d) <0$, it follows from definition of $f_{\infty}$ that $f(x+td_0)=-\infty$ for some $x\in\domsf f$,  which further indicates that $\inf (f+\frac{1}{2}\|A\cdot - b\|^2) = -\infty$ and $X=\varnothing$, which is (trivially) compact. On the other hand, if $f_{\infty}(d)>0$, $\forall d\in \kersf A\backslash\{0\}$, $X$ is compact (and non-empty) by Proposition \ref{p_cmp}-(vi).

(iii)-(iv) clear by contradiction. Indeed, (iv) is contrapositive of (iii).
\end{proof}

The converse statement of  Corollary \ref{c_cmp_p1}-(ii) is not true.  Considering  $f: (0,+\infty) \times \R \mapsto \R: (x_1,x_2)\mapsto -\log x_1$, $A=\begin{bmatrix}
0 & 1 \end{bmatrix}$ for example, $X=\varnothing$ is compact, but $\kersf f_{\infty}\mcap \kersf A = [0, +\infty)\times \{0\} \ne \{(0,0) \}$. This shows that the compactness of $X$ and  $\kersf f_{\infty}\mcap \kersf A = \{0\}$ are not equivalent.

The solution existence is not connected to either $\inf f=-\infty$ or $\inf f>-\infty$, e.g., $f(x)=x$ or $e^x$, both of which have solutions. However, Corollary \ref{c_cmp_p1}-(iii) states the relation between solution existence and infimum over $\kersf A$. The negation and converse of Corollary \ref{c_cmp_p1}-(iii) are not true, i.e., $\displaystyle \inf_{\substack{ t>0    \\ d \in \kersf A } }  f(x+td) > -\infty$ does not necessarily imply $X\ne\varnothing$, and $X=\varnothing$ does not necessarily also imply $\displaystyle \inf_{\substack{ t>0    \\ d \in \kersf A } }  f(x+td) = -\infty$.
 Taking $f: \R^2 \mapsto \R: (x_1,x_2)\mapsto e^{x_1}$, $A=\begin{bmatrix}
0 & 1 \end{bmatrix}$ for example, we have $\displaystyle \inf_{\substack{ t>0    \\ d \in \kersf A } }  f(x+td) =0 > -\infty$, and $X=\varnothing$.

We now show that without the condition of $\displaystyle \inf_{\substack{ t>0    \\ d \in \kersf A } }  f(x+td)>-\infty$, $\forall x\in\domsf f$, $ \kersf f_{\infty}\mcap \kersf A = \{0\}$ is no longer a sufficient and necessary condition for solution existence and compactness. What is more, it cannot guarantee solution existence. Considering  $f: \R^2 \mapsto \R: (x_1,x_2)\mapsto x_1$, $A=\begin{bmatrix} 0 & 1 \end{bmatrix}$, we obtain $f_{\infty}=f$, and $ \kersf f_{\infty}\mcap \kersf A = \{0\}$ is satisfied, but $X= \varnothing$. This is because  $\displaystyle \inf_{\substack{ t>0    \\ d \in \kersf A } }  f(x+td) = -\infty$ ($\forall x\in\domsf f$), which violates this condition.

The main drawback with the existing results is that the existence and compactness are bundled together, such that one cannot determine whether the solution set is empty or non-compact (but non-empty), if conditions are violated. Considering  $A=\begin{bmatrix}
0 & 1 \end{bmatrix}$  and taking $f_1: \R^2\mapsto\R: (x_1,x_2)\mapsto e^{x_1}$  and  $f_2: \R^2\mapsto\R: (x_1,x_2)\mapsto \max\{0,x_1\}$ respectively, both examples satisfy $\inf f >-\infty$ and $\kersf f_{\infty}\mcap \kersf A =  (-\infty, 0] \times \{0 \} \ne \{ (0,0) \} $. However,  the solution set of $f_1$ is empty, but that of $f_2$ is non-empty but unbounded. Thus, it is essential to discriminate the solution existence from compactness, both the above examples are suitably handled by our results (see Sect. \ref{sec_eg_p1}).

\subsubsection{Restricted coercivity}
In this sequel, we further explore restricted coercivity. In particular, if $\kersf A=\{0\}$, the condition of $f_{\infty} (d)>0$ ($\forall d\in \kersf A\backslash\{0\}$) is vacuously true, and this reduces to the ordinary coercivity \cite[Definition 3.1.1]{teboulle_book}.  We here focus on the case that $A$ has a non-trivial kernel.  The following important theorem reveals what the restricted coercivity implies for the domain of the conjugate function.
\begin{theorem} \label{t_coercivity}
If $f$ is coercive on $\kersf A$, then the following hold. 

{\rm (i)} $B_\alpha(0)  \subseteq \ransf A^\top + (B_\alpha(0) \mcap \kersf A) \subseteq \domsf (f(x_0+\cdot)|_{\kersf A})^*$, $\forall x_0\in\domsf f$;

{\rm (ii)} $0\in \intsf\ \domsf (f(x_0+\cdot)|_{\kersf A})^*$, $\forall x_0\in\domsf f$;

{\rm (iii)} If $0\in \domsf f$, $B_\alpha(0) \subseteq \domsf (f|_{\kersf A})^*$, $0\in \intsf\ \domsf (f|_{\kersf A})^*$;

{\rm (iv)} $ (f_{\infty}) |_{\kersf A}  = \sigma_{ \cP_{\kersf A} (\domsf f^*)}|_{\kersf A}  = \sigma_{\domsf f^*} |_{\kersf A}$;

{\rm (v)}  $0\in \risf\ \cP_{\kersf A} (\domsf f^*) $. 
\end{theorem}

\begin {proof}
(i)-(ii) Define function $g$ as $g(d)=f(x_0+d)+\iota_{\{d\in\R^n: Ad=0 \}} (d) = f(x_0+d)+\iota_{\kersf A} (d)$, with any $x_0\in \domsf f$. Thus, $g = g|_{\kersf A} = f(x_0+\cdot) |_{\kersf A} $. By Definition \ref{def_cmp}-(i), for $d\in\kersf A$, we have $\liminf_{\|d\|\rightarrow \infty} \frac{f(x_0+d)}{\|d\|} >0$,
$\forall x_0 \in\domsf f $, i.e., $\liminf_{\|d\|\rightarrow \infty} \frac{g(d)}{\|d\|} >0 $. This implies, by \cite[Proposition 14.16]{plc_book}, that $g$ is coercive on $\kersf A$, which indicates that $\exists \alpha>0$ and $\beta\in\R$, such that $g\ge \alpha \|\cdot\| +\beta$, for any input in $\kersf A$.

Consider $g^* = (g|_{\kersf A})^* =  (f(x_0 +\cdot)|_{\kersf A})^* $ (with any $x_0\in \domsf f$), we have:
\begin{eqnarray}
g^*(u) &=& \sup_{d\in\domsf  g} \langle u|d\rangle
-g(d)  = \sup_{d\in\kersf A} \langle u_\ksf|d\rangle
-g(d)
\nonumber \\
&  \le & \sup_{d\in\kersf A} \|u_\ksf\|\cdot\|d\|
- \alpha \|d\| -\beta  
 = \sup_{d\in\kersf A} (\|u_\ksf\| - \alpha )\cdot 
  \|d\| -\beta.
  \nonumber
\end{eqnarray}
Observe that if $\|u_\ksf\|\le \alpha$, then $g^*(u)$ always has the finite value of $-\beta$. This indicates that $u\in\dom g^*$, as long as $\|u_\ksf\|\le \alpha$. This has no restriction for $u_\rsf$. Formally, we write
\[
B_\alpha(0)  \subseteq \ransf A^\top +(B_\alpha(0) \mcap \kersf A) \subseteq \domsf g^*, 
\]
where $B_\alpha(0) $ is a ball of radius $\alpha$, centered at 0. In particular, $\ransf A^\top \subseteq \domsf g^*$, and $0\in \intsf\ \domsf g^*
=\intsf\ \domsf (f(x_0+\cdot)|_{\kersf A})^*$, $\forall x_0\in\domsf f$.

(iii) take $x_0=0$ in (i) and (ii).

(iv) By a basic result of $f_{\infty} = \sigma_{\domsf  f^*}$ \cite[Theorem 2.5.4]{teboulle_book}, we develop:
\[
(f_{\infty}|_{\kersf A} ) (d) 
= (\sigma_{\domsf f^*} |_{\kersf A} ) (d) 
=  \sigma_{\domsf f^*}  (d), \quad \forall d\in \kersf A.
\]
For $d\in \kersf A$, it further becomes:
\[
 \sigma_{\domsf f^*}  (d)
 = \sup_{u\in \domsf f^*} \langle u|d\rangle 
 = \sup_{u_\ksf \in \cP_{\kersf A} (\domsf f^*) } 
 \langle u_\ksf|d\rangle 
= \big(\sigma_{ \cP_{\kersf A} (\domsf f^*)}|_{\kersf A}
\big) (d).
\]

(v) The condition of $f_{\infty} (d)>0$, $\forall d\in\kersf A\backslash\{0\}$ is (trivially) equivalent to $(f_{\infty})|_{\kersf A} (d) >0$, $\forall d\ne 0$. By (iii),  we obtain equivalently $\sigma_{ \cP_{\kersf A} (\domsf f^*)} (d) > 0$, $\forall d\ne 0$. This implies $0\in \risf\ \cP_{\kersf A} (\domsf f^*) $, which also indicates that $\exists \alpha>0$, such that $B_\alpha(0) \mcap \kersf A \subseteq \cP_{\kersf A} (\domsf f^*)$. In other words, for every $v\in \kersf A$, there always exist $\epsilon>0$ and $u\in \ransf A^\top$,  such that $ \epsilon v + u \in\domsf f^*$\footnote{More precisely, $\epsilon$ can be chosen as $\epsilon\in (0, \frac{\alpha}{\|v\|})$.}.
\end{proof}

Theorem \ref{t_coercivity}-(iii) and (iv) are extensions of the result of \cite[Theorem 1.3.2]{teboulle_book}. The proof of (i)  implicitly used the basic rule of   \cite[Proposition 13.23]{plc_book} that 
$g^* = (g|_{\kersf A})^* = (g|_{\kersf A})^* \circ \cP_{\kersf A}$.

Interestingly, the condition of $f_{\infty}(d)>0$  ($\forall d\in\kersf A\backslash\{0\}$) can also be easily obtained by Theorem \ref{t_coercivity}. Indeed, this condition can be rewritten as $(f(x_0+\cdot)|_{\kersf A} )_{\infty} (d) >0$, $\forall d\ne 0$, for some $x_0\in\domsf f$. By a basic result of $f_{\infty} = \sigma_{\domsf  f^*}$ \cite[Theorem 2.5.4]{teboulle_book}, we develop:
\[
(f(x_0+\cdot)|_{\kersf A})_{\infty} (d) 
=\sigma_{\domsf (f(x_0+\cdot)|_{\kersf A})^*} (d) 
\ge \sup_{u\in B_\alpha(0) \cap \kersf A} 
\langle  u|d\rangle >0,
\]
where the above inequalities are due to 
$B_\alpha(0) \mcap \kersf A \subseteq \domsf 
(f(x_0+\cdot)|_{\kersf A})^*$, and  $0\in  \intsf\ \domsf (f(x_0+\cdot)|_{\kersf A})^*$.

\subsubsection{Connections between existing results and ours}
By Propositions \ref{p_cmp_p1} and \ref{p_cmp}, we have individually shown that both 
$\big(  \partial f^* (   A^\top  r)   \big)_{\infty} \mcap \kersf A =\{0\}$ and  $\cR_f  \mcap \kersf A =\{0\}$ are sufficient and necessary conditions for solution existence and compactness. Can we directly show their mutual equivalence without resorting to the solution compactness?  This question is key to understand the connections between the existing results and ours. The answer lies in Theorem \ref{t_connect}. Before this, we put the following key lemma, which directly connects the recession cone of $\partial f^*$ to recession function of $f$.
\begin{lemma} \label{l_connect}
$f_{\infty}$ and $  \big(  \partial f^* (   A^\top  r)   \big)_{\infty}$ are connected via the following relations:

{\rm (i)} $ f_{\infty} (d)  \left\{ \begin{array}{ll}
=\big\langle A^\top r | d \big\rangle, & \textrm{if\ }
   d \in  \big(  \partial f^* (   A^\top  r)   \big)_{\infty}; \\
> \big\langle A^\top r | d \big\rangle, &
\textrm{if\ } d \notin  \big(  \partial f^* (   A^\top  r)   \big)_{\infty}.
\end{array} \right. $

{\rm (ii)} $f_{\infty}(d) = \big\langle A^\top r | d \big\rangle \Longleftrightarrow 
  d \in  \big(  \partial f^* (   A^\top  r)   \big)_{\infty}$.
 \end{lemma}
\begin{proof}
(i) Let $d \in  \big(  \partial f^* (   A^\top  r)   \big)_{\infty}$, then $
x^\star+td \in   \partial f^* (   A^\top  r)$, $\forall t>0$, since $x^\star\in  \partial f^* (   A^\top  r)$. This implies that  $A^\top r \in \partial f(x^\star+td)$, $\forall t>0$. By convexity of $f$, we have
$f(x^\star)  \ge f(x^\star+td) -t \langle A^\top r | d\rangle$. This is
$\langle A^\top r | d\rangle \ge \frac{ f(x^\star+td)-f(x^\star) } {t}$,  $\forall t>0 $. Taking supremum over $t>0$ yields
$ \langle A^\top r | d\rangle \ge 
\sup_{t>0}  \frac{ f(x^\star+td)-f(x^\star) } {t} 
= f_{\infty} (d)$ by \cite[Proposition 2.4-(i)]{vaiter_thesis}.

On the other hand, by $A^\top r\in\partial f(x^\star)$ and convexity of $f$, we have $f(x^\star+td)  \ge  f(x^\star) +t \langle A^\top r | d\rangle $. It is $ \langle A^\top r | d\rangle \le \frac{ f(x^\star+td)-f(x^\star) } {t}$,  $\forall t>0 $. Taking supremum over $t>0$ yields
$ \langle A^\top r | d\rangle \le \sup_{t>0}
 \frac{ f(x^\star+td)-f(x^\star) } {t} = f_{\infty} (d)$. Combining both sides proves the first line of (i).

If $d\notin \big(  \partial f^* (   A^\top  r)   \big)_{\infty}$, then   $\forall x_0  \in \partial f^*(A^\top r)$\footnote{Of course, a typical choice is $x_0=x^\star \in \partial f^*(A^\top r)$ by \eqref{a1}.},  $\exists t_0>0$, such that $x_0 +t_0 d\notin\partial f^*(A^\top r)$ and then
 $A^\top r \notin\partial f (x_0 +t_0 d)$. By definition of subdifferential, we have $f(x_0 ) < f(x_0  +t_0 d) +\langle A^\top r | -t_0d \rangle$. It is $\langle A^\top r | d \rangle  < \frac{ f(x_0 +t_0 d) -f(x_0 )} {t_0 }
\le f_{\infty} (d)$, which shows the second line of (i).

(ii) is an immediate consequence of (i).
\end{proof}

\vskip.1cm
Then, we give the following important result for the direct connection.
\begin{theorem} \label{t_connect}
For the problem \eqref{p1}, the following equalities hold.
\[
\big(  \partial f^* (   A^\top  r)   \big)_{\infty} \mcap \kersf A 
= \kersf f_{\infty} \mcap \kersf A
= \cR_f  \mcap \kersf A 
=(\slevsf_{x^\star} f )_{\infty} \mcap \kersf A.
\]
\end{theorem}
\begin{proof}
First, when writing the notation of $\partial f^* (   A^\top  r) $, we have implicitly admitted the existence of $r$ (and $x_\rsf^\star$) and viability of $\partial f^* \circ A^\top$, which further implies that the solution set $X\ne\varnothing$ by Lemma \ref{l_X_p1}. It also indicates in turn by Corollary \ref{c_cmp_p1}-(iv) that $\displaystyle \inf_{\substack{ t>0    \\ d \in \kersf A } }  f(x+td)>-\infty$, $\forall x\in\domsf f$. Then, the second equality follows by Lemma \ref{l_cmp}-(v), and the last equality is due to Lemma \ref{l_cmp}-(iii).

Regarding the first equality, since $\kersf A$ appears in each side of the equality, we in this proof always assume $d\in \kersf A$. First, if $d\in \big(  \partial f^* (   A^\top  r)   \big)_{\infty}$, the first line of Lemma \ref{l_connect}-(i) shows  $f_{\infty}(d) =  \langle A^\top r | d \rangle = 0$, since $d\in\kersf A$. This shows $d\in \kersf f_{\infty}$. And thus, $\big(  \partial f^* (   A^\top  r)   \big)_{\infty} \mcap \kersf A  \subseteq \kersf f_{\infty} \mcap \kersf A $.

On the contrary, if $d\notin \big(  \partial f^* (   A^\top  r)   \big)_{\infty}$, the second line of Lemma \ref{l_connect}-(i) shows  $f_{\infty}(d) >  \langle r|Ad\rangle =0$, since $d\in \kersf A$, and thus $d\notin \kersf f_{\infty}$. This shows that $\kersf A\backslash \partial f^* (   A^\top  r)   \big)_{\infty} \subseteq \kersf A\backslash \kersf f_{\infty}$. Combining both sides proves the first equality. 
\end{proof}

Thus, it is easy to show the direct equivalence of $\big(  \partial f^* (   A^\top  r)   \big)_{\infty} \mcap \kersf A =\{0\}$ to the existing results listed in Proposition \ref{p_cmp}, without using the solution uniqueness.  Last, we present another less important but interesting result.
\begin{corollary}
$\big(  \partial f^* (   A^\top  r)   \big)_{\infty} \mcap \kersf A 
\subseteq 
\big(  \partial f^* (   A^\top  r)   \big)_{\infty} \mcap \kersf f_{\infty}$, where the inclusion cannot be replaced by equality, unless $\rank A=1$.
\end{corollary}
\begin{proof}
If $d\in \big(  \partial f^* (   A^\top  r)   \big)_{\infty}\mcap \kersf A$,  $f_{\infty}(d) =   \langle r|Ad\rangle = 0$, by the first line of Lemma \ref{l_connect}-(i). This shows $d\in\kersf f_{\infty}$. Conversely, if $d\in \big(  \partial f^* (   A^\top  r)   \big)_{\infty}\mcap \kersf f_{\infty}$,  $f_{\infty}(d) =   \langle A^\top r|d\rangle = 0$. This only shows that $d\perp A^\top r$, but does not necessarily imply $d\in \kersf A$, unless $\rank A =1$. 
\end{proof}

\subsection{Uniqueness}
\subsubsection{Our results}
 The uniqueness of solution is given below.
\begin{proposition} \label{p_unique_p1}
The particular solution $x^\star$ to \eqref{p1} is unique, (i.e., $X=\{x^\star\}$),

{\rm (i)} if and only if $(\partial f^* (   A^\top  r)  -x_\rsf^\star ) \mcap \kersf A$ is a singleton;

{\rm (ii)} if and only if $(\partial f^* (   A^\top  r)  -x^\star ) \mcap \kersf A = \{0\}$;

 {\rm (iii)} if  $\cP_{\kersf A} (\partial f^* (   A^\top  r)  -x^\star ) = \{0\}$, or equivalently, $(\cP_{\kersf A} \circ \partial f^* ) (   A^\top  r) = \{  x_\ksf^\star \}$. Furthermore, this condition is equivalent to $\partial f^*(A^\top r) \perp \kersf A$, and thus $\partial f^*(A^\top r) \subseteq x_\ksf^\star + \ransf A^\top$.
\end{proposition}
\begin{proof}
It has been shown in Corollary \ref{c_unique}-(ii) that $x_\rsf^\star$ is always unique, and the uniqueness depends on $x_\ksf^\star$ only. 

(i) Theorem \ref{t_X_p1}-(ii). Actually, this singleton is the kernel component of this unique solution---$x_k^\star$.

(ii) Theorem \ref{t_X_p1}-(iii).

(iii) Applying $\cP_{\kersf A}$ to both sides of second line of \eqref{a1} yields $  x_\ksf^\star   \in \cP_{\kersf A} \partial f^* \big( A^\top  r\big)$, which leads to (iii). 
\end{proof}

It is stressed that the Proposition \ref{p_unique_p1}-(iii) is sufficient but not necessary condition, by Lemma \ref{l_recession}-(iii), since $(\partial f^* (   A^\top  r)  -x^\star ) \cap \kersf A \subseteq
\cP_{\kersf A} (\partial f^* (   A^\top  r)  -x^\star ) = 
\cP_{\kersf A} (\partial f^* (   A^\top  r) ) -x_\ksf^\star$.  This can be shown by a lasso example of \cite{zh_lasso}.
\begin{example} \label{eg_lasso}
Consider a lasso example \cite{zh_lasso}:
\[
f: \R^3\mapsto\R: x\mapsto \|x\|_1,\quad A=\begin{bmatrix}
1 &0&2 \\  0&2&-2\end{bmatrix},\quad
b=\begin{bmatrix} 1 \\  1 \end{bmatrix}. 
\]  
It has the unique solution $x^\star =\begin{bmatrix}
0  & 1/4  & 0 \end{bmatrix}^\top$, and thus, 
$A^\top r =  \begin{bmatrix}
1 &  1 &  1 \end{bmatrix}^\top$, and
$\partial (\|\cdot\|_1^*) (A^\top r) =[0,+\infty)\times [0,+\infty)\times [0,+\infty)$. Using  $\kersf A = \xi \begin{bmatrix}
-2 &  1 & 1 \end{bmatrix}^\top$, $\forall \xi\in\R$, the uniqueness coincides with $\partial (\|\cdot\|_1^*) (A^\top r) \mcap \kersf A =  \{0\}$, but $\cP_{\kersf A} \big( \partial  (\|\cdot\|_1^*) (A^\top r)
\big) $, representing the projection of positive orthant in $\R^3$ onto the line $\R(-2, 1, 1)$, is clearly multi-valued, and 
Proposition \ref{p_unique_p1}-(iii) is not met.  Furthermore, $(\cP_{\kersf A} \big( \partial  (\|\cdot\|_1^*) (A^\top r)
\big) )_{\infty} \backslash \{ 0\} \ne\varnothing$ in this case,  Proposition \ref{p_cmp_p1}-(ii) is also not satisfied. This shows that Propositions \ref{p_cmp_p1}-(ii) and \ref{p_unique_p1}-(iii) are not necessary for compactness and uniqueness.
\end{example}

\subsubsection{Developments of the existing results}
The discussions of uniqueness were often based on several typical cones \cite{venkat_geometry,ame_edge,fadili_iiima}, and a number of geometric measures (e.g., Gaussian width and statistical dimension) were further developed under specific settings.
\begin{definition} \label{def_unique}
{\rm (i) (descent cone) \cite[Definition 2.7]{ame_edge}, \cite[Definition 3.3]{fadili_iiima}} The descent cone of $f$ at $x_0 \in \domsf f$ is defined as $\cD_f(x_0) = \mathrm{cone} \big\{ x-x_0: f(x)\le f(x_0) \big\} = \R_+ (x-x_0)$, for any $x$ satisfying $f(x)\le f(x_0)$. 

{\rm (ii) (radial cone) \cite[Definition 2.5]{fadili_cone}} The radial cone of a set  $C\subset\R^n$ at $x_0 \in C$ is defined as $\cR_C(x_0) = \mathrm{cone} (C-x_0) = \R_+ (C-x_0)$. 

{\rm (iii) (tangent cone) \cite[Definition 2.54]{shapiro_book}} The tangent cone of a set  $C\subset\R^n$ at $x_0 \in C$ is a topological closure of the radial cone, and defined as $\Tsf_C(x_0) = \overline{\cR_C(x_0)} = \clsf\  \R_+ (C-x_0)$. This is always closed. 
\end{definition}

The descent cone of $f$ at $x_0\in\domsf f$ is the set of descent directions of the function $f$ at the point $x_0$, which is not closed in general. Radial cone $\cR_C(x_0)$ is a cone of feasible directions in $C$ at $x_0$.  Applying tangent cone to $\slevsf_{x_0} f$ (at $x_0$), we have $\Tsf_{\slevsf_{x_0} f } (x_0) = \clsf (\R_+ (z-x_0) )$ ($\forall z\in \slevsf_{x_0} f $), which is always closed by definition. The above concepts are essentially equivalent based on the following
\begin{lemma} \label{l_descent}
Given a solution $x^\star$ to \eqref{p1}, it holds that 

{\rm (i)} $\cD_f(x^\star) = \Tsf_{\mathrm{slev}_{x^\star} f} (x^\star)$.

{\rm (ii)} $\cD_f(x^\star) \mcap \kersf A = \Tsf_{\mathrm{slev}_{x^\star} f} (x^\star)  \mcap \kersf A = \cR_{\partial f^*(A^\top r)} (x^\star)\mcap \kersf A  = (\slevsf_{x^\star} f -x^\star)\mcap \kersf A$.
\end{lemma}
\begin{proof}
(i) Take $d \in \cD_f(x^\star)$, then by definition, $\exists \xi >0$, such that $f(x+\xi d)\le f(x^\star)$. This shows $x^\star+ \xi d\in \mathrm{slev}_{x^\star}f$ by definition of $\slevsf_{x^\star}f$. Let $z=x^\star+\xi d$ and $z\in \slevsf_{x^\star} f$, then by definition of $\Tsf_{\mathrm{slev}_{x^\star} f} (x^\star)$, we have $\eta  (z-x^\star) = \eta \xi d\in \Tsf_{\slevsf_{x^\star} f } (x^\star)$, $\forall \eta>0$. Taking $\eta=1/\xi$ yields $d\in  \Tsf_{\slevsf_{x^\star} f } (x^\star)$. 

Conversely, take  $d\in  \Tsf_{\slevsf_{x^\star} f } (x^\star)$, then  $\exists \xi>0$, such that $d=\xi(z-x^\star)$ for some $z\in  \slevsf_{x^\star}f$. This implies $z=x^\star+d/\xi$, and $f(z) = f(x^\star+d/\xi)\le f(x^\star)$ since $z\in \slevsf_{x^\star} f$. Then $d/\xi \in \cD_f(x^\star)$, and further $d\in \cD_f(x^\star)$, since $\cD_f(x^\star)$ is a cone.

(ii) Take $d \in  \cR_{\partial f^*(A^\top r)} (x^\star)\mcap \kersf A$, then $\exists \xi>0$, such that $x^\star +\xi d\in \partial f^*(A^\top r)$. This yields $A^\top r \in \partial f (x^\star +\xi d)$. By convexity of $f$, we have
$f(x^\star)  \ge f(x^\star+\xi d) -\xi \langle A^\top r | d\rangle
=  f(x^\star+ \xi d)$, since $d\in \kersf A$. This shows that $d\in \cD_f(x^\star)$.

Conversely, Take $d \in \cD_f(x^\star) \mcap \kersf A$, then by definition, $\exists \xi >0$, such that $f(x^\star + \xi d)\le f(x^\star)$. Since $x^\star$ is a solution, we indeed have $f(x^\star + \xi d) = f(x^\star)$ by Corollary \ref{c_unique}-(iv). Note also that $\langle d|A^\top r\rangle=0$, since $d\in\kersf A$. It follows that
\[
f(x^\star + \xi d) = f(x^\star) = \langle A^\top r|x^\star \rangle
-f^*(A^\top r) = \langle A^\top r|x^\star + \xi d\rangle
-f^*(A^\top r)
\]
which implies that $x^\star + \xi d \in \partial f^*(A^\top r)$, i.e.,
$d\in  \cR_{\partial f^*(A^\top r)} (x^\star)$. 
\end{proof}

The above concepts were used in the following well-known uniqueness conditions, and they are connected and unified by Lemma \ref{l_descent}.
\begin{proposition} \label{p_unique_ex}
The solution $x^\star$ to \eqref{p1} is unique, if and only if any of the following equivalent conditions holds.

{\rm (i)  \cite[Theorem 5.1]{vaiter_thesis}}  $ \mathsf{T}_{\mathrm{slev}_{x^\star} f} (x^\star)
\mcap \kersf A=\{0\}$, i.e., $A$ is injective on   $\Tsf_{\mathrm{slev}_{x^\star} f} (x^\star)$.

{\rm (ii)  \cite[Propositions 3.4 and 4.13-(i)]{fadili_iiima}} 
 $\cD_{f} (x^\star) \mcap \kersf A=\{0\}$.

{\rm (iii)  \cite[Theorem 3.1]{fadili_cone}}  $\cR_{\partial f^*(A^\top r)} (x^\star) \mcap \kersf A=\{0\}$.
\end{proposition}
\begin{proof}
Their mutual equivalence has been shown in Lemma \ref{l_descent}. We make more remarks of (iii) here. In fact, \cite[Theorem 3.1]{fadili_cone} presents a result for more general result than \eqref{p1}. The proof of \cite[Theorem 3.1]{fadili_cone} can be specialized and simplified for \eqref{p1} as follows.

If  $ (\cR_{\partial f^*(A^\top r)} (x^\star)\mcap \kersf A) \backslash \{0\} \ne \varnothing$, then $\exists 0 \ne w\in \cR_{\partial f^*(A^\top r)} (x^\star)\mcap \kersf A$, and $\exists \xi>0$, such that $x^\star +\xi w\in \partial f^*(A^\top r)$. Then $x^\star +\xi w$ is also a solution to \eqref{p1}. This argument also shows the converse statement (by contradiction).
\end{proof}

\begin{remark}
Proposition \ref{p_unique_ex}-(ii) has also been specialized to the case of atom norm in \cite[Proposition 2.1]{venkat_geometry}. 
\end{remark}

\subsubsection{Connections between existing results and ours}
Then, let us discuss the connection between existing and ours regarding uniqueness.
\begin{lemma} \label{l_connect_unique}
$(\partial f^*(A^\top r) -x^\star)\mcap \kersf A = (\slevsf_{x^\star} f - x^\star) \mcap \kersf A
 =\cD_f(x^\star) \mcap \kersf A$.
\end{lemma}
\begin{proof}
We now show $(\partial f^*(A^\top r) -x^\star)\mcap \kersf A\subseteq (\slevsf_{x^\star} f - x^\star) \mcap \kersf A
 \subseteq \cD_f(x^\star) \mcap \kersf A
 \subseteq (\partial f^*(A^\top r) -x^\star)\mcap \kersf A$.
 
Take $v\in (\partial f^*(A^\top r) -x^\star)\mcap \kersf A$, then, $x^\star +v$ is also a solution, then $f(x^\star+v)=f(x^\star)$, then $x^\star +v\in\slevsf_{x^\star} f$.

Second, take $v\in (\slevsf_{x^\star} f - x^\star) \mcap \kersf A $, then, $f(x^\star+v)\le f(x^\star)$ and $Ax^\star = A(x^\star +v)$. Since $x^\star$ is an optimal solution, we have $f(x^\star+v)=f(x^\star)$, and $x^\star +v$ is also a solution. This shows $v\in\cD_f(x^\star)$ by definition.

Third, take  $v\in \cD_f(x^\star) \mcap \kersf A $, then $f(x^\star+v)\le f(x^\star)$ and $Ax^\star = A(x^\star +v)$. Since $x^\star$ is an optimal solution, we have $x^\star +v$ is also a solution. This shows
$v\in \partial f^*(A^\top r) - x^\star$ by Theorem \ref{t_X_p1}-(iii).
\end{proof}

Lemma \ref{l_connect_unique} directly establishes the equivalence of ours to existing results in Proposition \ref{p_unique_ex}, without using the bridge of solution uniqueness.

\subsection{Examples} \label{sec_eg_p1}
The properties of lasso solutions can be obtained using our results.
\begin{corollary} [lasso solutions] \label{c_lasso}
If $f=\|\cdot\|_1$ in \eqref{p1}, the following hold.

{\rm (i)} The lasso solution set is given as $X=x_\rsf^\star + \big( \big(\cN_C( A^\top r )-  x_\rsf^\star \big) \mcap \kersf A \big)$, 
where  $\cN_C(A^\top r)$ denotes a normal cone of set $C = \{u\in\R^n: \|u\|_{\infty} \le 1 \}$ at the point $A^\top r$, and $x_\rsf^\star = A^\dagger \big( I_m +  (A\circ \cN_C \circ 
A^\top)^{-1} \big)^{-1}   (b) $.

{\rm (ii)} $X =x^\star +  \big( \big(\cN_C( A^\top r ) -  x^\star \big) \mcap \kersf A \big)$, if a particular solution $x^\star\in X$ is given.

{\rm (iii)} The lasso solution always exists for any $b\in\R^m$, i.e., $X\ne \varnothing$.

{\rm (iv)} $x^\star$ is unique, if and only if $(\cN_{ \{u\in\R^n: \|u\|_{\infty} \le 1 \}} (A^\top r) -x^\star ) \mcap \kersf A = \{0\}$.

{\rm (v)} $x^\star$ is unique, if  $\cP_{\kersf A} \big(\cN_{ \{u\in\R^n: \|u\|_{\infty} \le 1 \}} (A^\top r)  -x^\star \big) = \{0\}$.
\end{corollary}
\begin{proof}
(i)-(ii): In view of Theorem \ref{t_X_p1}, we have $f^*=(\|\cdot\|_1)^* = \iota_{ \{u\in\R^n: \|u\|_{\infty} \le 1 \}}$, and $\partial f^* = \partial  \iota_{ \{u\in\R^n: \|u\|_{\infty} \le 1 \}} = \cN_{ \{u\in\R^n: \|u\|_{\infty} \le 1 \}}$. 

(iii) In view of Theorem \ref{t_ex_p1}-(v) and $0\in \risf\ \ransf (\partial \|\cdot\|_1) = \intsf [-1,1]^n$.

(iv)-(v): by Proposition \ref{p_unique_p1}.
\end{proof}

For the uniqueness of lasso solution, it is difficult to obtain the concrete expression of $\cD_f(x^\star)$. However, our result in (v) provides a simple and clear geometric picture for analysis. This example demonstrates the superiority of our expression to the existing result of $\cD_f(x^\star)\mcap \kersf A=\{0\}$. 

Again, Corollary \ref{c_lasso}-(v) is sufficient but not necessary for solution uniqueness. It surely deserves further investigation of the lasso solution properties (especially from the geometric view), and establish connections with existing works of \cite{gilbert,zh_lasso,lasso_reload}, by invoking the support (and cosupport) of a particular solution $x^\star$. This is left for near future work.

\vskip.2cm
We further list several (toy) examples in following tables. 
\begin{itemize}
\item Example-1: $\min_{x_1,x_2\in\R} x_1 +\frac{1}{2}(x_2-1)^2$;   
\item Example-2: $\min_{x_1,x_2\in\R} e^{x_1} +\frac{1}{2}(x_2-1)^2$;   
\item Example-3: $\min_{x_1,x_2\in\R} \max \{x_1, 0\} + \frac{1}{2}(x_2-1)^2$;   
\item Example-4: $\min_{x_1,x_2\in\R} \max\{ |x_1|-1, 0\}  +\frac{1}{2}(x_2-1)^2$;   
\item Example-5: $\min_{x_1,x_2\in\R} \max\{ e^{x_2} -x_1, 0 \}  +\frac{1}{2} x_1^2$  \cite[Remark 2.2]{bredies_ppa}; 
\item Example-6: $\min_{x_1,x_2\in\R} \max\{ e^{x_2} -x_1, 0 \}  +\frac{1}{2} x_2^2$.
\end{itemize}

These examples support our presented results, by providing several key insights as below.

$\bullet$ Example-1 shows that $\kersf f_{\infty} \subseteq \cR_f$ generally, while other examples show that $\kersf f_{\infty} =\cR_f$, if $\inf f>-\infty$.

$\bullet$ In Example-1, $\kersf f_{\infty} \mcap \kersf A=\{(0,0)\}$. However, the solution set is empty, this demonstrates the importance of the condition $\inf f>-\infty$.

$\bullet$ $\inf f>-\infty$ is over-sufficient for solution existence (assuming other compulsory conditions are fulfilled). However, $\inf f=-\infty$ does not always imply empty solution set. If Example-1 is changed to $f(x_1,x_2)=x_2$, we have $X=\R\times \{0\} \ne \varnothing$, though $\inf f=-\infty$. This is because $\displaystyle \inf_{\substack{ t>0    \\ d \in \kersf A } }  f(x+td) = x_2 >-\infty$. That is why we propose the infimum over $\kersf A$. 

$\bullet$ The violation of $\kersf f_{\infty} \mcap \kersf A=\{0\}$ may have different implications: the solution set is either empty or unbounded. It remains uncertain as to which of these two situations holds true. Comparing Examples-2 and 3 and another pair of Examples-5 and 6,  $\kersf f_{\infty} \mcap \kersf A\ne \{(0,0)\}$. This violation leads to non-existence in Examples-2 and 5, but unboundedness in Examples-3 and 6.

$\bullet$ Examples-3 and 6 show that the condition of $0\in \risf\ (\ransf \partial f - \ransf A^\top)$ is not necessary for the maximality. Yet, Example-5 shows the tightness of this condition. 

$\bullet$ Example-5 is adapted from \cite[Remark 2.2]{bredies_ppa}, demonstrating that the monotonicity of $A \circ \partial f^* \circ A^\top$ is not necessarily maximal monotone, even though $\partial f^*$ is maximally monotone. Consequently, the solution existence depends on the specific value of $b$. Specifically, $X \ne \varnothing$ if $b \in (0, +\infty)$, and $X = \varnothing$ otherwise. For any $b \in (0, +\infty)$, we always have $x_\rsf^\star = (b, 0)$ and $A^\top r = (0, 0)$, with $\partial f^*(A^\top r) = \{ (p, q): p \ge e^q \} = \{ (p, q): p > 0,\ q \le \log p \}$. Moreover, $X = \{ 0 \} \times (-\infty, \log b]$. This example illustrates that even when the monotonicity of $A \circ \partial f^* \circ A^\top$ is not maximal, the problem may still admit solutions (cf. Lemma \ref{l_ex_p1}).

$\bullet$ A comparison between Examples 5 and 6 demonstrates that altering $A$ can alter the maximality of monotonicity of $A \circ \partial f^* \circ A^\top$, since the interaction between $\partial f^*$ and $A$ is specific to their combination.

\begin{table}[h!]
  \centering
   \caption{Evaluation criteria for solution properties of \eqref{p1}.}\vspace{-1em}
   \resizebox{1.0\columnwidth}{!} {
 \begin{tabular}{|l||l|l|} 
    \hline
    criteria & Example-1 & Example-2 \\
    \hline\hline
fitting into \eqref{p1}  & 
    \tabincell{l}{ $f: (x_1,x_2)\mapsto x_1$  \\ 
  $A= \begin{bmatrix}
  0 & 1   \end{bmatrix} $ \\ $b= 1$  }  
     &  \tabincell{l}{ $f: (x_1,x_2)\mapsto e^{x_1}$  \\ 
  $A= \begin{bmatrix}
  0 & 1   \end{bmatrix} $ \\ $b= 1$  }   \\
     \hline \hline
{\bf existence} & \faTimes   & \faTimes  \\
     \hline 
$x_\rsf^*$ & --- & ---   \\
     \hline 
$X$ & $\varnothing$ & $\varnothing$   \\
     \hline 
$f^*: (u_1,u_2)\mapsto$ & $\iota_{ \{(1,0) \} }$ & 
$\left\{ \begin{array}{ll}
0, & \textrm{if\ }  (u_1,u_2)=(0,0); \\
u_1(\log u_1 -1), &\textrm{if\ }  (u_1,u_2)\in (0,+\infty)\times \{0\} ;\\
\varnothing,  & \textrm{otherwise} .
\end{array}\right.$  \\
     \hline 
$\partial f: (x_1,x_2)\mapsto $ & $\{ (1,0) \}$ 
& $\{ (e^{x_1}, 0) \}$   \\
     \hline 
$\partial f^*: (u_1,u_2)\mapsto$ & 
$\cN_{\{(1,0) \} } = \left\{ \begin{array}{ll}
\R^2, & \textrm{if\ }  (u_1,u_2)=(1,0); \\
\varnothing,  & \textrm{otherwise} .
\end{array}\right.$
 & $\left\{ \begin{array}{ll}
\{\log u_1 \} \times \R, & \textrm{if\ }  (u_1,u_2) \in (0,+\infty)\times \{0\}; \\
\varnothing,  & \textrm{otherwise} .
\end{array}\right.$    \\
     \hline 
$\ransf \partial f \mcap \ransf A^\top$ & $\varnothing$ 
&  $\varnothing$  \\
     \hline 
$0\in \risf\ \ransf \partial f$ & \faTimes &\faTimes   \\
     \hline \hline
{\bf compactness}  & \faTimes & \faTimes  \\  
     \hline 
$f_{\infty}: (d_1,d_2)\mapsto$  & $d_1\ (=f)$ 
& $\left\{ \begin{array}{ll}
0 , & \textrm{if\ }  (d_1,d_2) \in (-\infty,0]\times \R; \\
+\infty,  & \textrm{otherwise} .
\end{array}\right.$ \\
     \hline 
$\kersf  f_{\infty}$ & $\{0\}\times \R$ &  $(-\infty,0] \times \R$    \\
     \hline 
$\cR_f$ & $(-\infty,0] \times \R$ &  $(-\infty,0] \times \R$       \\
     \hline 
$\kersf  f_{\infty}\cap \kersf A$ & $\{ (0,0) \} $ 
& $ (-\infty,0]\times \{0\}$    \\
     \hline 
$\cR_f \cap \kersf A$ &$(-\infty,0] \times \{0\}$  
& $ (-\infty,0]\times \{0\}$      \\

     \hline \hline
{\bf uniqueness}  & \faTimes & \faTimes   \\
     \hline 
$ x^\star / \cD_f (x^\star) /\Tsf_{\slevsf_{x^\star} f} (x^\star)  $ 
  & --- &---      \\
     \hline 
    \end{tabular}
   }
    \vskip 0.5em
\label{table_reg}
\vspace*{-0.3cm}
\end{table}

\begin{table}[h!]
  \centering
   \caption{Evaluation criteria for solution properties of \eqref{p1}.}\vspace{-1em}
   \resizebox{1.0\columnwidth}{!} {
 \begin{tabular}{|l||l|l|} 
    \hline
    criteria & Example-3 & Example-4 \\
    \hline\hline
fitting into \eqref{p1}  & 
    \tabincell{l}{ $f: (x_1,x_2)\mapsto \max\{x_1,0\}$  \\ 
  $A= \begin{bmatrix}
  0 & 1   \end{bmatrix} $ \\ $b= 1$  }  
     &  \tabincell{l}{ $f: (x_1,x_2)\mapsto \max\{ |x_1|-1, 0\}$  \\ 
  $A= \begin{bmatrix}
  0 & 1   \end{bmatrix} $ \\ $b= 1$  }   \\
     \hline \hline
{\bf existence} & \faCheck   & \faCheck  \\
     \hline 
$x_\rsf^*$ & $(0,1)$ & $(0,1)$   \\
     \hline 
$X$ & $(-\infty,0] \times \{1\}$ & $[-1,1] \times \{1\}$   \\
     \hline 
$f^*: (u_1,u_2)\mapsto$ & $\left\{ \begin{array}{ll}
0, & \textrm{if\ }  (u_1,u_2)\in [0,1]\times \{0\}; \\
+\infty,  & \textrm{otherwise} .
\end{array}\right.$  
& $\left\{ \begin{array}{ll}
|u_1|, & \textrm{if\ }  (u_1,u_2)\in [-1,1]\times \{0\}; \\
+\infty,  & \textrm{otherwise} .
\end{array}\right.$ 
 \\
     \hline 
$\partial f: (x_1,x_2)\mapsto $ & 
$ \left\{ \begin{array}{ll}
\{ (1,0)\}, & \textrm{if\ } (x_1,x_2)\in (0,+\infty)\times \R ;\\
\{ (0,0)\}, & \textrm{if\ }  (x_1,x_2)\in (-\infty, 0)\times \R; \\ { }
[0,1]\times \{0\}, & \textrm{if\ } (x_1,x_2)\in \{0\}\times \R.
\end{array}\right. $ 
& $ \left\{ \begin{array}{ll}
\{ (-1,0)\}, & \textrm{if\ } (x_1,x_2)\in (-\infty, -1)\times \R ;\\ {}
[-1,0]\times \{ 0\}, & \textrm{if\ } (x_1,x_2)\in \{-1\}\times \R ;\\
\{ (0,0)\}, & \textrm{if\ } (x_1,x_2)\in (-1, 1)\times \R ;\\ {}
[0, 1]\times \{ 0\}, & \textrm{if\ } (x_1,x_2)\in \{1\}\times \R ;
\\
\{ (1,0)\}, & \textrm{if\ } (x_1,x_2)\in (1,+\infty)\times \R ;\\ \end{array}\right. $   \\
     \hline 
$\partial f^*: (u_1,u_2)\mapsto$ & 
$\left\{ \begin{array}{ll}
(-\infty, 0] \times \R, & \textrm{if\ }  (u_1,u_2)=(0,0); \\
\{ 0\} \times \R, & \textrm{if\ }  (u_1,u_2)\in (0,1) \times \{0\}; \\ { }
[0, +\infty) \times \R, & \textrm{if\ }  (u_1,u_2)=(1,0); \\
\varnothing, & \textrm{otherwise}.
\end{array}\right.$
 &  $ \left\{ \begin{array}{ll}
 (-\infty,-1]\times \R, & \textrm{if\ } (u_1,u_2) =(-1,0) ;\\ {}
\{-1\}\times \R, & \textrm{if\ } (u_1,u_2)\in (-1,0)\times \{0\};\\ { }
[-1, 1]\times \R, & \textrm{if\ } (u_1,u_2) = (0,0) ;\\ {}
 \{1\}\times \R & \textrm{if\ } (u_1,u_2)\in (0, 1) \times \{ 0\} ;
\\ { }
[1,+\infty)\times \R  , & \textrm{if\ } (x_1,x_2) = (1,0); \\
\varnothing, & \textrm{otherwise}. \end{array}\right. $     \\
     \hline 
$\ransf \partial f \mcap \ransf A^\top$ & $\{(0,0) \}$  
&  $\{(0,0) \}$ \\
     \hline 
$A\circ \partial f^* \circ A^\top: y\mapsto$ 
& $\left\{ \begin{array}{ll}
\R, & \textrm{if\ } y=0; \\
\varnothing,  & \textrm{otherwise} .
\end{array}\right.$    
& $\left\{ \begin{array}{ll}
\R, & \textrm{if\ } y=0; \\
\varnothing,  & \textrm{otherwise} .
\end{array}\right.$   \\
     \hline 
$\ransf (I_m  +A\circ \partial f^* \circ A^\top)$ & $\R$  & $\R$  \\
     \hline 
maximality of $A\circ \partial f^* \circ A^\top$ & \faCheck
 & \faCheck  \\
     \hline 
$\risf\ \ransf \partial f \mcap \ransf A^\top$ 
& $\varnothing$ &  $\{(0,0) \}$  \\
     \hline 
$0\in \risf\ \ransf \partial f$ & \faTimes & \faCheck   \\
     \hline \hline
{\bf compactness}  & \faTimes & \faCheck  \\  
     \hline 
$f_{\infty}: (d_1,d_2)\mapsto$  & $\max\{d_1,0\} \ (=f)$ 
& $|d_1|$ \\
     \hline 
$(\mathrm{slev}_{x} f )_{\infty}/\kersf  f_{\infty} /\cR_f$ 
&$(-\infty,0] \times \R$   & $ \{0\} \times \R$      \\
     \hline 
$\kersf  f_{\infty}(/\cR_f)\cap \kersf A$ &  $(-\infty,0] \times \{0\} $
& $ \{ (0,0) \}$    \\
     \hline 
$\partial f^*(A^\top r)$ & $(-\infty,0] \times \R$   
& $ [-1,1] \times \R$     \\
     \hline 
$(\partial f^*(A^\top r))_{\infty} \mcap \kersf A$ 
&    $(-\infty,0] \times \{0\} $ &  $ \{ (0,0)\} $    \\
     \hline 
$\cP_{\kersf A} \big( (\partial f^*(A^\top r))_{\infty} \big) $ 
& $(-\infty,0] \times \{0\}$   &  $ \{ (0,0)\} $     \\

     \hline \hline
{\bf uniqueness}  & \faTimes & \faTimes   \\
        \hline 
$ x^\star$ (chosen) &  $(0,1)$ & $(0,1)$  \\
        \hline 
$ \mathrm{slev}_{x^\star} f$ 
&  $(-\infty,0] \times \R $ &  $[-1,1] \times \R$    \\
        \hline 
$\mathsf{T}_{\mathrm{slev}_{x^\star} f} (x^\star) / \cD_f (x^\star)$ 
&  $(-\infty,0] \times \R $ &  $ \R^2$    \\
        \hline 
$\mathsf{T}_{\mathrm{slev}_{x^\star} f} (x^\star)
(/ \cD_f (x^\star)) \cap \kersf A$  &  $(-\infty,0] \times \{0\} $ 
&  $ \R  \times \{0\} $    \\
     \hline 
$ \cD_f (x^\star) \mcap \kersf A $ 
   &  $(-\infty,0] \times \{0\} $ 
&  $ \R  \times \{0\} $    \\
     \hline 
$ (\partial f^*(A^\top r) -x_\rsf^\star) \mcap \kersf A $ 
 &   $(-\infty,0] \times \{0\} $  &    $[-1,1] \times \{0\} $     \\
     \hline 
$(\cP_{\kersf A} \circ  \partial f^* ) (A^\top r) $ 
  &   $(-\infty,0] \times \{0\} $ & $[-1,1] \times \{0\} $    \\

 \hline
    \end{tabular}
   }
    \vskip 0.5em
\label{table_reg}
\vspace*{-0.3cm}
\end{table}

\begin{table}[h!]
  \centering
   \caption{Evaluation criteria for solution properties of \eqref{p1}.}\vspace{-1em}
   \resizebox{1.0\columnwidth}{!} {
 \begin{tabular}{|l||l|l|} 
    \hline
    criteria & Example-5 & Example-6 \\
    \hline\hline
fitting into \eqref{p1}  & 
    \tabincell{l}{ $f: (x_1,x_2)\mapsto \max\{e^{x_2}-x_1,0\}$  \\ 
  $A= \begin{bmatrix}
  1 & 0   \end{bmatrix} $ \\ $b=0$  }  
     &  \tabincell{l}{ $f: (x_1,x_2)\mapsto \max\{e^{x_2}-x_1, 0\}$  \\    $A= \begin{bmatrix}
  0 & 1   \end{bmatrix} $ \\ $b= 0$  }   \\
     \hline \hline
{\bf existence} & \faTimes   & \faCheck  \\
     \hline 
$x_\rsf^*$ & ---  & $(0, 0)$   \\
     \hline 
$X$ & $\varnothing$  & $[1,+\infty)\times \{0\}$   \\
     \hline 
$f^*: (u_1,u_2)\mapsto$ & \multicolumn{2}{l|}{ $\left\{ \begin{array}{ll}
u_2 \log(-\frac{u_2}{u_1}) - u_2, & \textrm{if\ }  (u_1,u_2)\in (-\infty, 0)\times (0, +\infty); \\
0, & \textrm{if\ }  (u_1,u_2) = (0,0); \\ 
+\infty,  & \textrm{otherwise} .
\end{array}\right.$  }
 \\
     \hline 
$\partial f: (x_1,x_2)\mapsto $ & 
\multicolumn{2}{l|}{ $ \left\{ \begin{array}{ll}
\{ (-1, e^{x_2})\}, & \textrm{if\ } e^{x_2} >x_1 ;\\
\{ (0, 0)\}, & \textrm{if\ } e^{x_2} <x_1 ;\\
\{(-t, te^{x_2}): t\in [0,1]\}, & \textrm{if\ } e^{x_2} =x_1 .
\end{array}\right. $ }  \\
     \hline 
$\partial f^*: (u_1,u_2)\mapsto$ & 
\multicolumn{2}{l|}{  $\left\{ \begin{array}{ll}
(-\infty, v) \times \{\log v\}, & \textrm{if\ }  (u_1,u_2) \in \{-1\} \times (0, +\infty); \\
\{ (-\frac{u_2}{u_1}, \log (-\frac{u_2}{u_1}) )\} \times \R, & \textrm{if\ }  (u_1,u_2)\in (-1, 0) \times  (0, +\infty); \\ { }
\{(p,q): p\ge e^q \}, & \textrm{if\ }  (u_1,u_2)=(0,0); \\
\varnothing, & \textrm{otherwise}.
\end{array}\right.$ }   \\
     \hline 
$\ransf \partial f \mcap \ransf A^\top$ & $\{(0,0) \}$  
&  $\{(0,0) \}$   \\
     \hline 
$A\circ \partial f^* \circ A^\top: y\mapsto$ 
& $\left\{ \begin{array}{ll}
(0, +\infty), & \textrm{if\ } y=0; \\
\varnothing,  & \textrm{otherwise} .
\end{array}\right.$    
& $\left\{ \begin{array}{ll}
\R, & \textrm{if\ } y=0; \\
\varnothing,  & \textrm{otherwise} .
\end{array}\right.$       \\
     \hline 
$\ransf (I_m  +A\circ \partial f^* \circ A^\top)$ & 
$(0, +\infty)$ & $\R$  \\
     \hline 
maximality of $A\circ \partial f^* \circ A^\top$ & \faTimes
 & \faCheck  \\
     \hline 
$\risf\ \ransf \partial f \mcap \ransf A^\top$ 
& $\varnothing$ &  $\varnothing$  \\
     \hline 
$0\in \risf\ \ransf \partial f$ & \multicolumn{2}{l|}{ \faTimes } \\
     \hline \hline
{\bf compactness}  & \faTimes & \faCheck  \\  
     \hline 
$f_{\infty}: (d_1,d_2)\mapsto$  & \multicolumn{2}{l|}{ 
$\left\{ \begin{array}{ll}
0, & \textrm{if\ } (d_1,d_2)\in [0,+\infty)\times (-\infty, 0]; \\
-d_1, & \textrm{if\ } (d_1,d_2)\in (-\infty, 0)\times (-\infty, 0]; \\
+\infty, & \textrm{if\ } (d_1,d_2)\in \R \times (0, +\infty).
\end{array}\right.$     } \\
     \hline 
$(\mathrm{slev}_{x} f)_{\infty} / \kersf  f_{\infty} / \cR_f $ 
&  \multicolumn{2}{l|}{ $ [0,+\infty)\times (-\infty, 0]$  }    \\
     \hline 
$\kersf  f_{\infty} (/\cR_f) \cap \kersf A$ &  $ \{0\} \times (-\infty, 0]$ 
& $ [0,+\infty)\times \{0\} $     \\
     \hline 
$\partial f^*(A^\top r)$ & --- & $ \{(p,q): p\ge e^q \} $     \\
     \hline 
$(\partial f^*(A^\top r) )_{\infty}$ & --- & $[0, +\infty)
\times (-\infty,0] $     \\
     \hline 
$(\partial f^*(A^\top r))_{\infty} \mcap \kersf A$ 
& --- &  $[0, +\infty) \times \{0\} $      \\
     \hline 
$\cP_{\kersf A} \big( (\partial f^*(A^\top r))_{\infty} \big) $ 
& ---  &  $[0, +\infty) \times \{0\} $       \\

     \hline \hline
{\bf uniqueness}  & \faTimes & \faTimes   \\
     \hline 
$ x^\star $ (chosen) & ---  &  $(1, 0) $      \\
     \hline 
$\slevsf_{x^\star} f$ & ---  & $ \{(p,q): p\ge e^q \} $   \\
     \hline 
$\Tsf_{\slevsf_{x^\star} f} (x^\star) / \cD_f (x^\star)$ & --- 
 & $ \{(p,q): p\ge q \} $   \\
     \hline 
$ \Tsf_{\slevsf_{x^\star} f} (x^\star) (/\cD_f (x^\star)) \mcap \kersf A $ 
   &  --- & $[0,+\infty)\times\{0\}$   \\
     \hline 
$ (\partial f^*(A^\top r) -x_\rsf^\star) \mcap \kersf A $ 
 & ---  &  $[1, +\infty) \times \{0\} $      \\
     \hline 
$(\cP_{\kersf A} \circ  \partial f^* ) (A^\top r) $ 
  & --- & $(0, +\infty) \times \{0\} $       \\

 \hline
    \end{tabular}
   }
    \vskip 0.5em
\label{table_reg}
\vspace*{-0.3cm}
\end{table}

\section{Equality-constrained minimization problem} \label{sec_p2}
\subsection{Solution set}
Similar to Sect. \ref{sec_rls}, we first give  preliminary expressions of solution set of the problem \eqref{p2} in Lemma \ref{l_X_p2}, and then study the solution properties. Finally, Lemma \ref{l_X_p2} is refined to Theorem \ref{t_X_p2}.  
\begin{lemma} \label{l_X_p2}
The solution set of \eqref{p2} is given as

{\rm (i)} $X =   A^\dagger b  + \big(  \mcup_{v\in\R^m} \partial f^*(A^\top v ) -A^\dagger b \big) \mcap \kersf A $; 

{\rm (ii)} $X =   x^\star  + \big(  \mcup_{v\in\R^m} \partial f^*(A^\top v ) -x^\star \big) \mcap \kersf A $, if a particular solution $x^\star$ is given. 
\end{lemma}
\begin{proof}
(i) By subspace decomposition of $x=x_\rsf+x_\ksf$, \eqref{p2} is reformulated as $\min_{x_\rsf,x_\ksf\in\R^n} f(x_\rsf+x_\ksf)$, subject to $Ax_\rsf = b$ and $Ax_\ksf = 0$. Noting that $x_\rsf^\star = A^\dagger b$, it yields that $\min_{x_\ksf \in\R^n} f(A^\dagger b + x_\ksf)$, subject to $Ax_\ksf = 0$. 
Using indicator function, this constrained problem is equivalently rewritten as $\min_{x_\ksf\in\R^n} f(A^\dagger b + x_\ksf) +\iota_C (x_\ksf)$, 
where $C=\{x\in \R^n| Ax=0\}=\kersf A$.  Taking subdifferential w.r.t. $x_\ksf$, the Fermat's rule yields $0 \in \partial f(A^\dagger b + x_\ksf^\star) + \cN_{\kersf A}(x_\ksf^\star)$, 
where $\cN_{\kersf A}(x_\ksf^\star) $ denotes the normal cone of ${\kersf A}$ at $x_\ksf^\star$, and indeed $\cN_{\kersf A}(x_\ksf^\star) =\ransf A^\top$. Thus, it becomes
\be \label{e2}
\exists v_0 \in\R^m,\quad \textrm{s.t.\ }A^\top v_0 \in \partial f(A^\dagger b + x_\ksf^\star),\quad\textrm{where\ }
x_\ksf^\star\in \kersf A,
\ee
and then, $A^\dagger b + x_\ksf^\star \in \mcup_{v\in\R^m} \partial f^*(A^\top v )$, subject to $x_\ksf^\star\in \kersf A$. Thus, the kernel part is given as $x_\ksf^\star \in \big(  \mcup_{v\in\R^m} \partial f^*(A^\top v )
-A^\dagger b \big) \mcap \kersf A$, from which (i) follows.

(ii)  If $x^\star$ is a solution, it implies that $x_\rsf^\star = A^\dagger b$ and $x_\ksf^\star \in \big(  \mcup_{v\in\R^m} \partial f^*(A^\top v )
-A^\dagger b \big) \mcap \kersf A$. Subtracting $x_\ksf^\star$ on both sides, we obtain
\begin{eqnarray}
X &= &  A^\dagger b  + \big(  \mcup_{v\in\R^m} \partial f^*(A^\top v )
-A^\dagger b \big) \mcap \kersf A 
\nonumber \\
 & = &  A^\dagger b +x_\ksf^\star  + \Big( \big(  \mcup_{v\in\R^m} \partial f^*(A^\top v )
-A^\dagger b \big) \mcap \kersf A \Big) -x_\ksf^\star
\nonumber \\
 & = & x^\star + \big(  \mcup_{v\in\R^m} \partial f^*(A^\top v )
-A^\dagger b - x_\ksf^\star \big) \mcap 
\big( \kersf A-x_\ksf^\star \big)
\nonumber \\
 & = & x^\star + \big(  \mcup_{v\in\R^m} \partial f^*(A^\top v )
- x^\star \big) \mcap \kersf A.
\nonumber
\end{eqnarray}
\end{proof}

One can see from the above proof that the subspace decomposition is a very convenient tool for solving \eqref{p2}, since the equality constraint automatically determines the unique range component: $x_\rsf^\star = A^\dagger b$ (assuming $b\in\ransf A$). The kernel part $x_\ksf^\star$ can then be found by applying $\partial f^*$. From Lemma \ref{l_X_p2}, it is easy to find the following sufficient and necessary condition for solution existence with given $b\in \R^m$. 
\begin{theorem} \label{t_ex_p2}
For the problem \eqref{p2}, $X\ne \varnothing$, if and only if 
 $b\in \ransf  (A \circ \partial f^* \circ A^\top) = \domsf (A\infimal \partial f)$.
\end{theorem}
\begin{proof}
Observing \eqref{p2}, based on the proof of Lemma \ref{l_X_p2}, we develop
\begin{eqnarray}
&&  X\ne\varnothing 
\nonumber \\
& \  \Longleftrightarrow \ & 
\exists x_\rsf^\star\in \ransf A^\top, x_\ksf^\star \in\kersf A, \textrm{\ such that $Ax_\rsf^\star = b$ and \eqref{e2}  holds}
\nonumber \\
& \  \Longleftrightarrow \ &
\textrm{$b\in\ransf A$ and $\exists v_0\in\R^m$, $x_\ksf^\star\in \kersf A$, such that $A^\dagger b + x_\ksf^\star \in  \partial f^*(A^\top v_0 )$}
\nonumber \\
& \  \Longleftrightarrow \ &
\textrm{$b\in\ransf A$ and $\exists v_0\in\R^m$, such that $AA^\dagger b \in A \partial f^*(A^\top v_0 )$}
\nonumber \\
& \  \Longleftrightarrow \ &
\textrm{$\exists v_0\in\R^m$, such that $b \in  (A\circ  \partial f^* \circ A^\top)  v_0$ }
\nonumber \\
& \  \Longleftrightarrow \ &
\textrm{$b\in \ransf  (A\circ \partial  f^* \circ A^\top)
=\domsf (A\infimal \partial f)$. }
\nonumber
\end{eqnarray}
\end{proof}

The following result is an immediate consequence.
\begin{corollary} \label{c_exi_p2}
The solution set of \eqref{p2} is non-empty for any $b\in\ransf A$, if and only if $\ransf  (A \circ \partial f^* \circ A^\top) = \ransf A$.
\end{corollary}
\begin{proof}
Clear by Theorem \ref{t_ex_p2}. Alternatively, this can also be obtained from  Lemma \ref{l_X_p2}-(i). 
First, observe that $b\in\ransf A$ is necessary for the solution existence, to meet the hard constraint of \eqref{p2}.  From Lemma \ref{l_X_p2}, $X\ne \varnothing$ with any $b\in\ransf A$ requires that $\big(  \mcup_{v\in\R^m} \partial f^*(A^\top v ) - A^\dagger b \big) \mcap \kersf A \ne\varnothing$ holds for all $b\in\ransf A$. Since $A^\dagger b \in \ransf A^\top$, this further implies that the projection of the set  $ \mcup_{v\in\R^m} \partial f^*(A^\top v )$ onto $\ransf A^\top$ could cover the entire subspace $\ransf A^\top$. In other words, for any given $b_0 \in \ransf A$, there correspondingly exists at least $v_0\in\R^m$, such that $\cP_{\ransf A^\top} ( \partial f^*(A^\top v_0 ) ) \owns A^\dagger b_0$, i.e.,   $ (A\circ  \partial f^* \circ A^\top) ( v_0 )  \owns  b_0$. This is required to hold for any $b_0\in\ransf A$, thus, it is equivalent to say $\ransf (A\circ  \partial f^* \circ A^\top) \supseteq \ransf A$. Since  $\ransf (A\circ  \partial f^* \circ A^\top) \subseteq \ransf A$, we finally have $\ransf (A\circ  \partial f^* \circ A^\top) = \ransf A$.
\end{proof}

\begin{remark}
{\rm (i)} Recalling problem \eqref{p1}, the point $x_\ksf^\star$ exists automatically whenever $x_\rsf^\star$ exists. In comparison, the existence of $x_\rsf^\star = A^\dagger b$ for problem \eqref{p2} is independent of $x_\ksf^\star$. Hence, the existence of $x_\ksf^\star$ requires additional analysis. This constitutes a fundamental distinction between the two problems.

{\rm (ii)} Again, it is implicitly assumed by \eqref{e2} that $\ransf\partial f\cap \ransf A^\top \ne \varnothing$, which is a fundamental viability condition for the operators $A\circ \partial f^* \circ A^\top$ and $A\infimal \partial f$.
\end{remark}

Based on Theorem \ref{t_ex_p2} and Corollary \ref{c_exi_p2}, the solution set can be refined as follows.
\begin{theorem} \label{t_X_p2}
The solution set of \eqref{p2} is given as

{\rm (i)} $\displaystyle X =   A^\dagger b  + \big(  \mcup_{v\in (A\infimal \partial f) (b) } \partial f^*(A^\top v ) -A^\dagger b \big) \mcap \kersf A $.

{\rm (ii)} $X = x^\star  + \big(  \mcup_{v\in\R^m: A^\top v\in\partial f(x^\star)}  \partial f^*(A^\top v ) -x^\star \big) \mcap \kersf A $, if a particular solution $x^\star$ is given. 
\end{theorem}
\begin{proof}
(i) By Theorem \ref{t_ex_p2},  $X\ne \varnothing$ is equivalent to  $\exists v_0\in\R^m$, such that $b \in  (A\circ  \partial f^* \circ A^\top)  (v_0)$. This is $v_0\in  (A\circ  \partial f^* \circ A^\top)^{-1}  (b) = (A\infimal \partial f) (b)$. This shows that the union of $v$ does not need to take those outside of $(A\infimal \partial f) (b)$ into account.

(ii) If a solution $x^\star$ is provided, it implies that $X\ne\varnothing$ and $0\in  \mcup_{v\in\R^m} \partial f^*(A^\top v ) -x^\star$, 
i.e., $ x^\star  \in \mcup_{v\in\R^m} \partial f^*(A^\top v ) $. This implies that those $v \in\R^m$ satisfying $x^\star \notin \partial f^*(A^\top v)$ can be excluded in this union. In other words, it suffices to take $v$ satisfying $x^\star \in \partial f^*(A^\top v)$  into account. 
\end{proof}

Indeed, Theorem \ref{t_ex_p2} can itself be derived from Theorem \ref{t_X_p2}. Specifically, Theorem \ref{t_X_p2}-(i) states that $X = \varnothing$ if and only if $(A\infimal \partial f)(b) = \varnothing$, i.e., $b \notin \domsf (A\infimal \partial f) = \ransf (A \circ \partial f^* \circ A^\top)$. Regarding Theorem \ref{t_X_p2}-(ii), any $v \in \R^m$ satisfying $A^\top v \notin \partial f(x^\star)$ contradicts the optimality of $x^\star$ and must therefore be excluded. Conversely, for any $v_0 \in \R^m$ such that $A^\top v_0 \in \partial f(x^\star)$, it follows that $\partial f(x^\star) \mcap \ransf A^\top \ne \varnothing$; this is known as the {\it source condition}, and such a vector $A^\top v_0$ is referred to as a {\it dual certificate} \cite{certificate}.

\subsection{Solution properties}
We begin by analyzing the compactness and uniqueness of the solution to \eqref{p2}, and then examine the influence of the hard constraint of $Ax=b$ on the minimization of $f$.

\subsubsection{Compactness and uniqueness}
The results are as follows.
\begin{proposition} \label{p_cmp_p2}
The solution set $X$ of \eqref{p2} is non-empty and compact, if and only if one of the following equivalent conditions holds:

{\rm (i) [unknown $x^\star$]} $ \big(  \mcup_{v\in (A\infimal \partial f) (b) } \partial f^*(A^\top v ) \big)_{\infty} \mcap \kersf A  = \{0\}$; 

{\rm (ii) [known $x^\star$]} $\big(  \mcup_{v\in\R^m: A^\top v\in \partial f(x^\star)} \partial f^*(A^\top v ) \big)_{\infty} \mcap \kersf A = \{0\}$.
\end{proposition}
\begin{proof}
(i) is obtained based on Theorem \ref{t_X_p2}-(i) and by following the proof of Proposition \ref{p_cmp_p1}. Note that this condition has implied that $b\in\domsf (A\infimal \partial f)$, and thus $(A\infimal \partial f) (b) \ne \varnothing$.

(ii) follows from Theorem \ref{t_X_p2}-(ii).
\end{proof}

\begin{proposition} \label{p_unique_2}
The solution to \eqref{p2} is unique, if and only if  one of the following equivalent conditions holds:

{\rm (i) [unknown $x^\star$]} $\big(  \mcup_{v\in (A\infimal \partial f) (b) } \partial f^*(A^\top v ) - A^\dagger b \big) \mcap \kersf A$ is singleton;
 
{\rm (ii) [known $x^\star$]} $\big(  \mcup_{v\in\R^m: A^\top v\in \partial f(x^\star)} \partial f^*(A^\top v ) - x^\star  \big) \mcap \kersf A = \{0\}$.
\end{proposition}
\begin{proof}
Clear by Theorem \ref{t_X_p2}.
\end{proof}

\subsubsection{Influence of the constraint $Ax=b$}
\eqref{p2} is a constrained optimization problem. Can we determine whether the constraint $Ax = b$ affects the minimization of $f$? In other words, if $\min f$ exists, is $(A\infimal f) (b)$ strictly greater than $\min f$?
The solution set provided in Theorem \ref{t_X_p2} offers a criterion to determine whether $Ax = b$ influences the minimization of $f$.
\begin{proposition}  \label{p_affect}
Assuming both of constrained \eqref{p2} and unconstrained $\min f$ have solutions, the constraint of $Ax=b$ has no effect upon the minimization of $f$ (i.e., $(A\infimal f) (b) = \min_{x: Ax=b} f(x) = \min f$), if and only if any of the following equivalent conditions holds:

{\rm (i)}  $ \big( \partial f^*(0) - A^\dagger b \big) \mcap \kersf A
\ne \varnothing$;

{\rm (ii)} $b\in (A\circ  \partial f^*) (0)$.
\end{proposition}
\begin{proof}
(i) Considering $\min_{x\in\R^n} f(x)$ with $f\in\Gamma_0(\R^n)$, we have, by Fermat's rule, that $0\in \partial f(x^\star)$, i.e., $x^\star\in \partial f^*(0)$. 
Recalling \eqref{e2}, if $0\in \partial f(A^\dagger b+x_\ksf)$ for some $x_\ksf\in\kersf A$, then the optimal value of \eqref{p2} keeps the same as $\min f$. The condition is translated as: $\exists x_\ksf\in\kersf A$, such that $0\in \partial f(A^\dagger b+x_\ksf)$, i.e., $A^\dagger b+ x_\ksf \in \partial f^*(0)$. 

On the contrary, if $0\notin \partial f(A^\dagger b+x_\ksf)$, $\forall x_\ksf\in\kersf A$, then the optimal value of \eqref{p2} must be (strictly) greater than $\min f$. This condition is equivalent to $A^\dagger b+x_\ksf \notin \partial f^*(0)$, $\forall x_\ksf \in \kersf A$, i.e., $(A^\dagger b+\kersf A)\mcap \partial f^*(0) =\varnothing$. 

(ii) We develop  the equivalence: 
\[
 \big( \partial f^*(0) - A^\dagger b \big) \mcap \kersf A\ne \varnothing
\Longleftrightarrow
A^\dagger b \in \cP_{\ransf A^\top} ( \partial f^*(0) )
\Longleftrightarrow
 b \in A ( \partial f^*(0) ),
 \]
 since the solution existence of \eqref{p2} has already assumed that $b\in\ransf A$.
\end{proof}


\subsection{Related works and discussions}
Most of the existing works discussed the solution uniqueness of \eqref{p2} based on descent cone and radial cone. We now summarize and unify these results and further show the equivalence of our result with them.

\begin{lemma} \label{l_unique_1}
If $x^\star$ is an optimal solution to \eqref{p2}, then the following hold.

{\rm (i)  \cite[Proposition 4.4]{fadili_cone}} 
$\kersf A\mcap \cD_f(x^\star) = \kersf A\mcap \cR_{\partial f^*(y)}(x^\star)$, $\forall y\in \partial f(x^\star) \mcap \ransf A^\top $.

{\rm (ii)}  $\cR_{\partial f^*(y)}(x^\star) =   \mcup_{v\in\R^m: A^\top v\in \partial f(x^\star)} \partial f^*(A^\top v ) - x^\star$, 
$\forall y\in  \partial f(x^\star) \mcap \ransf A^\top $.
\end{lemma}
\begin{proof}
(i) See the proof of \cite[Proposition 4.4]{fadili_cone}.

(ii) We develop
\begin{eqnarray}
&& x \in \cR_{\partial f^*(y)}(x^\star) , \quad  
\forall y\in  \partial f(x^\star) \mcap \ransf A^\top
\nonumber \\
  & \Longleftrightarrow &
x \in \cR_{\partial f^*(A^\top v)}(x^\star) , \quad  
\forall v \textrm {\ satisfying\ } A^\top v\in \partial f(x^\star) 
\nonumber \\
&  \Longleftrightarrow &
 x \in \R_+  \big( \partial f^*(A^\top v ) -x^\star \big), \quad \forall v\in\R^m: A^\top v\in\partial f(x^\star)
\nonumber \\
&  \Longleftrightarrow &
x \in \xi  \big(  \mcup_{v\in\R^m: A^\top v\in \partial f(x^\star)} \partial f^*(A^\top v ) -x^\star \big), \quad \forall \xi \ge 0
\nonumber \\
&  \Longleftrightarrow &
x \in  \mcup_{v\in\R^m: A^\top v\in \partial f(x^\star)} \partial f^*(A^\top v ) -x^\star, \quad \textrm{taking\ } \xi = 1.
\nonumber
\end{eqnarray}
\end{proof}

Then the solution uniqueness of \eqref{p2} is summarized below.
\begin{theorem}
The solution $x^\star$ to \eqref{p2} is unique, if and only if any of the following equivalent conditions holds.

{\rm (i) \cite[Proposition 4.3]{fadili_cone},  \cite[Proposition 3.4]{fadili_iiima}} $\kersf A\mcap \cD_f(x^\star) = \{0\}$. 

{\rm (ii)} $\big(  \mcup_{v\in\R^m: A^\top v\in\partial f(x^\star)}  \partial f^*(A^\top v ) -x^\star \big) \mcap \kersf A = \{0\}$
\end{theorem}
\begin{proof}
Clear by either Lemma \ref{l_unique_1} or Proposition \ref{p_unique_2}-(ii). We here provide alternative proof by contradiction.

Assuming $\big( \big(  \mcup_{v\in\R^m: A^\top v\in\partial f(x^\star)}  \partial f^*(A^\top v ) -x^\star \big) \mcap \kersf A \big) \backslash \{0\} \ne \varnothing$, take $0\ne w\in \big(  \mcup_{v\in\R^m: A^\top v\in\partial f(x^\star)}  \partial f^*(A^\top v ) -x^\star \big) \mcap \kersf A$. Then 
$\exists v_0\in\R^m$, suth that $x^\star +w \in  \partial f^*(A^\top v_0 )$ and $w\in \kersf A$, which implies 
$A^\top v_0 \in \partial f(x^\star +w)$. This indicates 
$f(x^\star) \ge  f(x^\star+w)+\langle A^\top v_0 |-w\rangle
= f(x^\star+w)$. This shows that $x^\star +w$ is also a solution to 
\eqref{p2}, which contradicts to the uniqueness of $x^\star$. 
\end{proof}

\subsection{Infimal postcomposition}
\subsubsection{Exactness of infimal postcomposition}
It is easy to recognize that the problem \eqref{p2} has close connections with exactness of infimal postcomposition of $f$ by $A$, which is given as \cite[Definition 12.34]{plc_book}:
\[
A\infimal f: b \mapsto \inf_{x\in\R^n} f(x),\quad 
\textrm{s.t.\ }\ Ax=b.
\] 

The following facts regarding $A\infimal f$ will be useful in our expositions. Refer to Appendices \ref{app_2} and \ref{app_3} for the detailed proof.
\begin{fact} \label{f_infimal}
Given $f\in \Gamma_0(\R^n)$ and a linear mapping $A:\R^n\mapsto \R^m$, the following hold.

{\rm (i)}  $f^*\in\Gamma_0(\R^n)$, and $f^*\circ A^\top$ is convex and l.s.c., but not necessarily proper. 

{\rm (ii)} $f^*\circ A^\top$ is proper, and furthermore, $f^*\circ A^\top \in \Gamma_0(\R^m)$, if and only if $\ran A^\top \mcap \dom f^* \ne \varnothing$.

{\rm (iii)}  $A\triangleright f$ is convex, but not necessarily proper and closed.     

{\rm (iv)} If $A\triangleright f$ is exact, this is proper (and convex), but not necessarily closed.     

{\rm (v)} $\partial (A\infimal f)$ is not necessarily maximally monotone.

{\rm (vi)} $\partial (A\infimal f)$ is  maximally monotone, if $A\triangleright f$ is exact and closed.
\end{fact}

\begin{fact} \label{f_dom}
If  $A\infimal f\in\Gamma_0(\R^m)$, then the following hold.

{\rm (i)}  $\domsf (A\infimal \partial f) 
 \subseteq  A ( \domsf \partial f) \subseteq A(\domsf f)$;
 
{\rm (ii)} $\risf \ \domsf (A\infimal f) 
\subseteq \domsf \partial  (A \infimal f) \subseteq \domsf (A\infimal f) = A(\domsf f)$.
\end{fact}

Obviously, the exactness of $A\infimal f$ at a point $b \in \R^m$ is equivalent to the solution existence of \eqref{p2} with this $b$. Thus, 
Theorem \ref{t_ex_p2} can be translated to the conditions for the exactness of $A\infimal f$.
\begin{proposition} \label{p_ex_infimal_2}
Assuming $A\infimal f\in\Gamma_0(\R^m)$, then  $A\infimal f$ is exact

{\rm (i)}  at a point $b\in\R^m$, if and only if
$b \in \domsf (A\infimal \partial f)$.

{\rm (ii)} in $A(\domsf f)$, if and only if
$\domsf (A\infimal \partial f) = A(\domsf f)$.

{\rm (iii)} in $\domsf \partial (A\infimal  f)$, if
$\domsf (A\infimal \partial f) = A(\domsf f)$.
\end{proposition}
\begin{proof}
(i)-(ii) Clear by Theorem \ref{t_ex_p2} and Fact \ref{f_dom}-(i).

(iii) Clear by Fact \ref{f_dom}-(ii).
\end{proof}

\subsubsection{Relation between both problems}
The infimal postcomposition also provides an interesting link between the regularized least-squares  \eqref{p1} and its constrained counterpart \eqref{p2}. 

First, the solution existence of \eqref{p1} can also be associated with the exactness of $A\infimal f$.
\begin{proposition} \label{p_both}
For \eqref{p1} with given $b \in\R^m$, $X\ne\varnothing$, if and only if 

{\rm (i)} $b\in \ransf \big(I_m +\partial (A\infimal f) \big)$;

{\rm (ii)} $A\infimal f$ is exact at $J_{\partial (A\infimal f)} (b)$.
\end{proposition}
\begin{proof}
We take a different approach from the proof of Theorem \ref{t_X_p1}-(ii) as below:
\begin{eqnarray}
&& X \ne\varnothing  \textrm{\ for \eqref{p1} with given $b\in\R^m$}
\nonumber \\
  & \Longleftrightarrow &
\Arg\min_{t,x} f(x)+\frac{1}{2} \big\| t - b \big\|^2 +\iota_{ \{ x: Ax=t\} } (x) \ne\varnothing
\nonumber \\
&  \Longleftrightarrow &
 \Arg \min_{t\in\R^m} \underbrace{ \Big( 
\min_{x\in\R^n}  f(x)+ \iota_{ \{ x: Ax=t\} } (x) \Big) }_{:=(A\infimal f)(t)} 
+ \frac{1}{2}\|t-b\|^2 \ne\varnothing
\nonumber \\
&  \Longleftrightarrow &
\left\{ \begin{array}{l}
\textrm{$t^\star = ( I_m +\partial (A\infimal f) )^{-1} (b)$ exists for  $b\in\R^m$}  \\
\textrm{$x^\star$ exists for this $t^\star$}  \\
\end{array} \right.
\nonumber \\
&  \Longleftrightarrow &
\left\{ \begin{array}{l}
b\in \ransf  ( I_m +\partial (A\infimal f) )   \\
A\infimal f \textrm{\ is exact at $t^\star = J_{\partial (A\infimal f)} (b)$} 
\end{array} \right.
\nonumber 
\end{eqnarray}

It is emphasized that although $J_{\partial (A\infimal f)} (b)$ satisfies $J_{\partial (A\infimal f)} (b) \in \domsf \partial (A\infimal f) $, it is not guaranteed that $J_{\partial (A\infimal f)} (b) \in \domsf (A\infimal \partial f)$. This is the reason why condition (ii) is necessary. Also note that the resolvent $J_{\partial (A\infimal f)} $ may not have a full domain, since the condition $A\infimal f\in\Gamma_0(\R^m)$ does not necessarily hold. There is no need to enforce $A\infimal f\in \Gamma_0(\R^m)$,  since it is sufficient but not necessary for the maximal monotonicity of $\partial (A\infimal f)$.
\end{proof}

\begin{theorem} \label{t_both}
The solution set to \eqref{p1} is non-empty for any $b\in\R^m$, if and only if 

{\rm (i)} $\partial (A \triangleright f)$ is maximally monotone.

{\rm (ii)} $A\infimal f$ is exact in $\domsf \partial (A\infimal f)$;
\end{theorem}
\begin{proof}
Clear by the proof of Proposition \ref{p_both}, Fact \ref{f_infimal}-(v),   convexity of $A\infimal f$ (by Fact \ref{f_infimal}-(ii)) and Minty's theorem \cite[Theorem 21.1]{plc_book}.
\end{proof}

\begin{remark}
It is known by Proposition \ref{p_ex_infimal_2}-(i) that $A\infimal f$  is exact in $\domsf (A\infimal \partial f)$. However, requiring that $A\infimal f$ be exact on $\domsf \partial (A\infimal f)$, as in Theorem \ref{t_both}-(ii), is subtle due to the difficulty in comparing $\domsf (A\infimal \partial f)$ and $\domsf \partial (A\infimal f)$.
\end{remark}

It is also interesting to reexpress the solution set of \eqref{p1} in terms of that of \eqref{p2}.
\begin{corollary}
The solution set of \eqref{p1} can be  alternatively expressed as
\[
X = A^\dagger t^\star + \big(  \mcup_{v\in (A\infimal \partial f) (t^\star) } \partial f^*(A^\top v ) - A^\dagger t^\star \big) \mcap \kersf A, 
\]
where $t^\star  = J_{\partial (A\infimal f)} (b)$.
\end{corollary}
\begin{proof}
Combine Proposition \ref{p_both}, Theorems \ref{t_both} and \ref{t_X_p2}. 
\end{proof}

The well-known stronger condition  $0\in \risf\ (\ransf A^\top -\ransf \partial f)$  (cf. Theorem \ref{t_ex_p1}-(iv)) would resolve all these issues and ensure that many conditions are satisfied.
\begin{corollary}
If $0\in \risf\ (\ransf A^\top -\ransf \partial f)$, or equivalently, $\risf\ \ransf\partial f\mcap \ransf A^\top \ne\varnothing$, then the following hold.

{\rm (i)} $A\infimal f = (f^*\circ A^\top)^* \in\Gamma_0(\R^m)$;

{\rm (ii)} $\partial (A \triangleright f) = A\infimal \partial f = (A\circ \partial f^* \circ A^\top)^{-1} = \partial (f^* \circ A^\top)^* $ is maximally monotone;

{\rm (iii)} $\risf\ \domsf (A\infimal f) \subseteq \domsf \partial (A \infimal f) = \domsf (A\infimal \partial f) \subseteq \domsf (A\infimal f) =A(\domsf f)$;

{\rm (iv)}  $A\infimal f$ is exact almost everywhere in $A(\domsf f)$; 

{\rm (v)} \eqref{p2} has solutions for almost all $b\in A(\domsf f)$;

{\rm (vi)} \eqref{p1} has solutions for any $b\in\R^m$;
\end{corollary}
\begin{proof}
(i)-(ii): In view of \cite[Corollary 13.40, Corollary 15.28-(i), Proposition 16.42, Corollary 16.53-(i), Proposition 25.41-(iv), Corollary 25.44]{plc_book}. 

(iii) follows from (ii).

(iv) is due to (iii) and Theorem \ref{t_ex_p2}. 

(v) is equivalent to (iv). 

(vi) is due to either (ii) and Theorem \ref{t_ex_p1}-(iii), or 
 (iv) and Proposition \ref{p_both}.
\end{proof}

\subsection{Examples}
We list several (toy) examples in following tables, which support our presented results.
\begin{itemize}
\item Example-1: $\min_{x\in\R} -\log x $, s.t.  $x=0$;   
\item Example-2: $\min_{x_1,x_2\in\R} e^{x_1}$, s.t. $x_1 = 0$;   
\item Example-3: $\min_{x_1,x_2\in\R} e^{x_1}$, s.t. $x_2 = 0$;   
\item Example-4: $\min_{x_1,x_2\in\R} \max\{ e^{x_2} -x_1, 0 \}$, s.t. $x_1 = 0$;   
\item Example-5: $\min_{x_1,x_2\in\R} \max\{ e^{x_2} -x_1, 0 \}$, s.t. $x_2 = 0$.
\end{itemize}

\begin{table}[h!]
  \centering
   \caption{Evaluation criteria for solution properties of \eqref{p2}.}\vspace{-1em}
   \resizebox{1.0\columnwidth}{!} {
 \begin{tabular}{|l||l|l|l|} 
    \hline
    criteria & Example-1 & Example-2 & Example-3 \\
    \hline\hline
fitting into \eqref{p2}  & 
    \tabincell{l}{ $f =-\log$  \\  $A=1$ \\ $b=0$  }  
     &  \tabincell{l}{ $f: (x_1,x_2)\mapsto e^{x_1}$  \\    $A= \begin{bmatrix}
   1 & 0   \end{bmatrix} $ \\ $b= 0$  } 
     &  \tabincell{l}{ $f: (x_1,x_2)\mapsto e^{x_1}$  \\    $A= \begin{bmatrix}
  0 & 1   \end{bmatrix} $ \\ $b= 0$  }  \\
     \hline \hline
{\bf existence} & \faTimes   & \faCheck & \faTimes  \\
     \hline 
$x_\rsf^*$ & $(0,0)$  & $(0, 0)$ & $(0, 0)$    \\
     \hline 
$X$ & $\varnothing$  & $\{0\}\times \R$ & $\varnothing$    \\
     \hline 
$A(\domsf f)$ & 
$(0,+\infty)$ & $\R$ & $\R$   \\
     \hline 
$\partial f^*$ & 
 $\left\{ \begin{array}{ll}
-\frac{1}{u}, & \textrm{if\ }  u \in (-\infty, 0); \\
\varnothing,  & \textrm{otherwise}.
\end{array}\right.$   
 &  \multicolumn{2}{l|}{ $\left\{ \begin{array}{ll}
\{\log u_1\}\times \R, & \textrm{if\ } (u_1,u_2)\in 
(0,+\infty)\times \{0\}; \\
\varnothing,  & \textrm{otherwise}.
\end{array}\right.$ }  \\
     \hline 
$A \circ \partial f^*\circ A^\top $ & 
 $\left\{ \begin{array}{ll}
-\frac{1}{v}, & \textrm{if\ } v \in (-\infty, 0); \\
\varnothing,  & \textrm{otherwise} .
\end{array}\right.$   
 & $\left\{ \begin{array}{ll}
\{\log v\}, & \textrm{if\ } v \in (0,+\infty); \\
\varnothing,  & \textrm{otherwise}.
\end{array}\right.$
 & $\varnothing$   \\
     \hline 
$\domsf (A \infimal \partial f) $ & 
$(0,+\infty)$& $\R$ & $\varnothing$   \\
     \hline 
$\ransf  \partial f \cap \ransf A^\top $ & 
$(-\infty, 0)$ & $(0,+\infty)\times \{0\}$ & $\varnothing$   \\
     \hline 
$b\in \domsf (A \infimal \partial f)  $ & \faTimes   & \faCheck & \faTimes   \\
     \hline 
exactness of $A \infimal f$ at $b$ & \faTimes   & \faCheck & \faTimes   \\
     \hline 
$\mcup_{v\in (A\infimal \partial f) (b)} \partial f^*( A^\top v) $ &
$\varnothing$ & $\partial f^*(A^\top 1)=\{0\}\times \R$ & $\varnothing$   \\
     \hline 
 $b\in (A \circ \partial f^*)(0)$ & \faTimes   & \faTimes & \faTimes   \\
     \hline 
 $(A\infimal f)(b) =\min f$ & \faTimes   & \faTimes & \faTimes   \\

     \hline \hline
{\bf uniqueness}   & \faTimes & \faTimes  & \faTimes \\
     \hline 
$ x^\star $ (chosen) & ---  &  $(0, 0) $   & ---   \\
     \hline 
$C:= \mcup_{v\in (A\infimal \partial f) (b)} \partial f^*( A^\top v) $ &
$\varnothing$ & $\partial f^*(A^\top 1)=\{  0\} \times \R $  &  $\varnothing$\\
     \hline 
$(C -x^\star)\mcap \kersf A $ &
$\varnothing$ & $\{  0\} \times \R$ & $\varnothing$ \\
 \hline 
    \end{tabular} }
    \vskip 0.5em
\label{table_reg}
\vspace*{-0.3cm}
\end{table}

\begin{table}[h!]
  \centering
   \caption{Evaluation criteria for solution properties of \eqref{p2}.}\vspace{-1em}
   \resizebox{1.0\columnwidth}{!} {
 \begin{tabular}{|l||l|l|} 
    \hline
    criteria & Example-4 & Example-5 \\
    \hline\hline
fitting into \eqref{p2}  & 
    \tabincell{l}{ $f: (x_1,x_2)\mapsto \max\{e^{x_2}-x_1,0\}$  \\ 
  $A= \begin{bmatrix}
  1 & 0   \end{bmatrix} $ \\ $b=0$  }  
     &  \tabincell{l}{ $f: (x_1,x_2)\mapsto \max\{e^{x_2}-x_1, 0\}$  \\    $A= \begin{bmatrix}
  0 & 1   \end{bmatrix} $ \\ $b= 0$  }   \\
     \hline \hline
{\bf existence} & \faTimes   & \faCheck  \\
     \hline 
$x_\rsf^*$ & $(0, 0)$  & $(0, 0)$   \\
     \hline 
$X$ & $\varnothing$  & $[1,+\infty)\times \{0\} $   \\
     \hline 
$A(\domsf f)$ & $\R$  & $\R$   \\
     \hline 
$A\circ \partial f^* \circ A^\top$ 
& $\left\{ \begin{array}{ll}
(0, +\infty), & \textrm{if\ } v=0; \\
\varnothing,  & \textrm{otherwise} .
\end{array}\right.$    
& $\left\{ \begin{array}{ll}
\R, & \textrm{if\ } v=0; \\
\varnothing,  & \textrm{otherwise} .
\end{array}\right.$       \\
     \hline 
$\domsf (A \infimal \partial f)$ & 
$(0, +\infty)$ & $\R$  \\
     \hline 
$b\in \domsf (A \infimal \partial f)$ & 
\faTimes & \faCheck  \\
     \hline 
exactness of $A \infimal f$ at $b$ & \faTimes   & \faCheck   \\
     \hline 
$\mcup_{v\in (A\infimal \partial f) (b)} \partial f^*( A^\top v) $ &
$\varnothing$ & $\{(p,q): p\ge e^q \}$    \\
     \hline 
 $b\in (A \circ \partial f^*)(0)$ & \faTimes   & \faCheck    \\
     \hline 
 $(A\infimal f)(b) =\min f$ & \faTimes   & \faCheck   \\
     \hline \hline
{\bf uniqueness}   & \faTimes & \faTimes   \\
     \hline 
$ x^\star $ (chosen) & ---  &  $(1, 0) $      \\
     \hline 
$C:= \mcup_{v\in (A\infimal \partial f) (b)} \partial f^*( A^\top v) $ &
$\varnothing$ & $\partial f^*(A^\top 0)=\{ (p,q): p\ge e^q\}$  \\
     \hline 
$(C -x^\star)\mcap \kersf A $ &
$\varnothing$ & $[1,+\infty)\times \{ 0\}$  \\
 \hline
    \end{tabular}
   }
    \vskip 0.5em
\label{table_reg}
\vspace*{-0.3cm}
\end{table}

\section{Concluding remarks}
This work provided a fundamental understanding of the solution properties of regularized least-squares and its constrained version. Many geometric and analytical concepts have been connected, unified and further developed.

As mentioned in Sect. \ref{sec_eg_p1}, the geometric properties of lasso solutions deserve further investigations based on our Corollary \ref{c_lasso} and informative Example \ref{eg_lasso}. Its connections to existing works, e.g., \cite{lasso_reload,zh_lasso, gilbert,vaiter_geometry}, need to be established. Moreover, this work  can be easily extended to the case of $f(Dx)+\frac{1}{2}\|Ax-b\|^2$, where the function $f$ is composed with a linear (analysis) operator $D$. This helps to understand generalized lasso solutions. Finally, observing that the objective function of regularized least-squares is strongly convex with respect to the range component of $x$ (or projection of $x$ onto some transformed subspace), it would be of (theoretical) interest to perform dimension reduction and devise accelerated algorithms based on this observation.


\appendix
\section{Proof of Lemma \ref{l_recession}}
\label{app_1}
\begin{proof}
(i) First, take $v\in C_{\infty} \mcap D$, then $tv\in C_{\infty} \mcap D$, $\forall t>0$ (since $D$ is a linear subspace). Since $C\mcap D\ne \varnothing$, take $x\in C\cap D$, we have $x+tv\in C$ (since $v\in C_{\infty}$) and  $x+tv\in D$ (since $D $ is a subspace). 
This implies  $x+tv\in C\mcap D$, $\forall t>0$, and thus, $v\in (C\mcap D)_{\infty}$ by definition.

Conversely, take $v\in (C\mcap D)_{\infty}$ and $x\in C\mcap D$. Then $x+tv \in C\mcap D$, $\forall t>0$, by definition. It implies that $x+tv\in C$ and $x+tv \in D$. Thus, $v\in C_{\infty} \mcap D$, since $D$ is a linear subspace.

(ii) Since $C$ is closed and convex, $C_{\infty}(x)$ does not depend on specific point $x\in C$ \cite[Proposition III-2.2.1]{urruty}. By \cite[Theorem 8.1]{rtr_book}, the recession cone of $C$ can also be expressed as $C_{\infty} = \{d\in\R^n: C+d\subseteq C\}$.
Then, we deduce the following equivalence: $d\in C_{\infty} \Longleftrightarrow C+d \subseteq C \Longleftrightarrow (C-x_0) +d\subseteq C-x_0 \Longleftrightarrow d\in (C-x_0)_{\infty}$. 

(iii) $\cP_{D} (C)$ is defined as $\cP_{D} (C) = \big\{x\in \R^n: \exists v\in C, \textrm{\ s.t.\ } x = \cP_D v \big\}$. Take $x\in C\mcap D$, then $x = \cP_D x$ (since $x\in D$) for $x\in C$. Thus, $x\in \cP_D(C)$. This inclusion of (iii) is often proper. Consider a counter-example in $\R^2$: $C=[1,2]\times [1,2] \subset \R^2$ and $D=\R\times \{0\}$, which satisfies $C\mcap D=\varnothing\subset \cP_D (C) = [1,2] \times \{0\}$. 
\end{proof}

\section{Proof of Fact  \ref{f_infimal}}
\label{app_2}
\begin{proof}
(i): $f^*\in\Gamma_0(\R^n)$ is due to \cite[Corollary 13.38]{plc_book}. $f^*\circ A^\top$ is convex and closed by \cite[Proposition 9.5]{plc_book}, but not necessarily proper.   Considering $f(x_1,x_2) = x_1$ and $A=\begin{bmatrix}
0 & 1 \end{bmatrix}$ for example, we have $f^* = \iota_{ \{ (1,0) \} }$, and   $f^*\circ A^\top \equiv +\infty$, which is improper.

(ii) Clear. Furthermore, the viability condition of $\ransf \partial f\mcap \ransf A^\top \ne\varnothing$ is sufficient to ensure $f^*\circ A^\top \in\Gamma_0(\R^m)$, since $\domsf f^*\mcap \ransf A^\top \supseteq \domsf \partial f^*  \mcap \ransf A^\top =\ransf \partial f\mcap \ransf A^\top \ne\varnothing$ by   \cite[Corollary 16.49]{plc_book}.

(iii) The convexity of $A\triangleright f$ is due to \cite[Proposition 12.36-(ii)]{plc_book}. Regarding the properness,  the example in (i) leads to  $A\infimal f\equiv -\infty$, which is improper. To show that $A\infimal f$ is not necessarily closed, consider 
\be \label{y5}
 f:\R^2\mapsto \R:  (x_1, x_2) \mapsto   \begin{cases} 
  0, & \text{if } x_1 \geq 0, \, x_2 \geq 0, \, x_1x_2 \geq 1; \\
  +\infty, & \text{otherwise},
  \end{cases}, \quad A=\begin{bmatrix}
0 & 1 \end{bmatrix}.
\ee
We have $ A \triangleright f =\iota_{  (0,+\infty)   }$, which is not closed, since $  (A \triangleright f)(0) = +\infty$, but $\lim_{k \to \infty} (A \triangleright f)(1/k) = 0$.
Moreover, $A\infimal f$ is exact in $A(\domsf f) = (0,+\infty)$. This also shows that the exactness of $A\infimal f$ could not guarantee its closedness.

(iv) We first show the properness of $A\infimal f$, if $A\infimal f$ is exact. Since $f\in\Gamma_0(\R^n)$, \(\exists x_0 \in \domsf f \), such that \( f(x_0) < +\infty \). Let \( y_0 = Ax_0 \), then
$ (A \triangleright f)(y_0) \leq f(x_0) < +\infty$.
On the other hand, since \( f \) is proper, \( f(x) > -\infty \) for all \( x \). Thus, $  (A \triangleright f)(y) \geq \inf_{x} f(x) > -\infty$, \(\forall y \in \domsf (A\infimal f) \), 
which ensures that \( A \triangleright f \) never takes \(-\infty\). On the other hand, the example of \eqref{y5} shows that $A\infimal f$ may not be closed, even if this is exact.

(v) Recalling the example of \eqref{y5}, we have  
$\partial (A\infimal f) (y) = 
\begin{cases}
\{0\},       & \text{if } y > 0, \\
\varnothing,   & \text{if } y \leq 0.
\end{cases}$. Its closure function is  $\mathrm{cl} (A \triangleright f) =\iota_{ [0,+\infty)   }$, and the subdifferential at 0 is
$\partial(\mathrm{cl}(A\infimal f))(0) = \{0\}$. This shows that \( \partial (A\infimal f) \) is not maximally monotone, since  \( \partial (A\infimal f) \subset \partial(\mathrm{cl}(A\infimal f)) \), and the graph of the subdifferential of  \( \partial (A\infimal f) \) lacks \( (0, 0) \).

(vi) If $A\infimal f$ is closed, combining with (iv), it yields that $A\infimal f\in\Gamma_0(\R^m)$, then $\partial (A\infimal f)$ is maximally monotone by \cite[Theorem 20.25]{plc_book}. 
\end{proof}

\section{Proof of Fact  \ref{f_dom}}
\label{app_3}
\begin{proof}
(i) Given  $A\infimal f\in\Gamma_0(\R^m)$, based on \cite[Corollary 16.49]{plc_book}, we develop
\begin{eqnarray}
\domsf (A\infimal \partial f) &=& \ransf  (A \circ \partial f^* \circ A^\top) = A \big( \ransf (\partial f^* \circ A^\top) 
\big)  \quad \textrm{(by definition of $A\infimal \partial f$)}
\nonumber \\
& \subseteq & A ( \ransf \partial f^* ) = 
A ( \domsf \partial f)  
\quad \textrm{(by $\partial f^* = (\partial f)^{-1}$)}
\nonumber \\
&\subseteq & A(\domsf f).
\quad \textrm{(by \cite[Corollary 16.49]{plc_book})}
\nonumber
\end{eqnarray}

(ii) The first two inclusions are from \cite[Proposition 16.27, Corollary 16.49]{plc_book}, and the last equality is due to \cite[Proposition 12.36-(i)]{plc_book}.
\end{proof}

\bibliographystyle{siam}

\small{
\bibliography{refs}
}

\end{document}